\documentclass[11pt]{amsart}

\usepackage[colorlinks = true,
            linkcolor = red,
            urlcolor  = magenta,
            citecolor = black,
            anchorcolor = blue]{hyperref}
\usepackage{color}
\usepackage[shortalphabetic,initials,nobysame]{amsrefs}
\usepackage{upgreek}
\usepackage{tikz}
\usepackage[applemac]{inputenc} % for Macs
\usetikzlibrary{arrows}
\usepackage[all]{xy}
\usepackage{graphicx}
\usepackage{wrapfig}
\usepackage{yfonts}
\usepackage{graphicx}
\usepackage{amssymb}
\usepackage{epstopdf}
\usepackage{mathrsfs}
\usepackage[mathcal]{eucal}
\usepackage{bm}

\renewcommand{\eprint}[1]{\href{https://arxiv.org/abs/#1}{#1}}

\DeclareMathOperator{\GL}{\mathrm{GL}}

\DeclareMathOperator{\SL}{\mathrm{SL}}

\DeclareMathOperator{\diag}{diag}
\DeclareMathOperator{\Hom}{Hom}
\DeclareMathOperator{\Res}{Res}
\DeclareMathOperator{\Lie}{Lie}

\newcommand{\fg}{\frak{g}}

\BibSpec{article}{%
    +{}  {\PrintAuthors}                {author}
    +{,} { \textit}                     {title}
     +{,} {}                            {note}
    +{.} { }                            {part}
    +{:} { \textit}                     {subtitle}
    +{,} { \PrintContributions}         {contribution}
    +{.} { \PrintPartials}              {partial}
    +{,} { }                            {journal}
    +{}  { \textbf}                     {volume}
    +{}  { \PrintDatePV}                {date}
    +{,} { \issuetext}                  {number}
    +{,} { \eprintpages}                {pages}
    +{,} { }                            {status}
    +{,} { \DOI}                   {doi}
    +{,} { \eprint}        {eprint}
      +{,} {\publisher}                {publisher}
    +{,} { \address}                   {address}
    +{}  { \parenthesize}               {language}
    +{}  { \PrintTranslation}           {translation}
    +{;} { \PrintReprint}               {reprint}
    +{.} {}                             {transition}
    +{}  {\SentenceSpace \PrintReviews} {review}
}

\newtheorem{Thm}{Theorem}[section]
\newtheorem{Lem}[Thm]{Lemma}
\newtheorem{Prop}[Thm]{Proposition}
\newtheorem{Cor}[Thm]{Corollary}
\newtheorem{Con}[Thm]{Conjecture}

\theoremstyle{definition}
\newtheorem{Def}[Thm]{Definition}
\theoremstyle{remark}
\newtheorem{Rem}[Thm]{Remark}

\newtheoremstyle{named}{}{}{\itshape}{}{\bfseries}{.}{.5em}{#1 #3}
\theoremstyle{named}

\def\Q{\mathbb{Q}}

\def\C{\mathbb{C}}
\def\Z{\mathbb{Z}}

\def\P{\mathbb{P}}

\def\fb{\mathfrak{b}}
\def\g{\mathfrak{g}}
\def\Frenkel:2013uda{\mathfrak{h}}

\def\sl{\mathfrak{sl}}

\def\cE{\mathcal{E}}
\def\cF{\mathcal{F}}

\def\cL{\mathcal{L}}

\def\cO{\mathcal{O}}

\def\cV{\mathcal{V}}
\def\cW{\mathcal{W}}

\def\a{\alpha}

\def\ze{\zeta}

\def\bo{\textbf{o}}

\def\=>{\Longrightarrow}

\def\to{\longrightarrow}

\def\o+{\oplus}
\def\bo+{\bigoplus}

\def\<{\langle}
\def\>{\rangle}
\def\({\left(}
\def\){\right)}

\def\^{\wedge}
\def\+{\dagger}

\def\dd[#1,#2]{\frac{d#1}{d#2}}
\def\del[#1,#2]{\frac{\partial #1}{\partial #2}}
\def\over[#1]{\overline{#1}}
\def\vec[#1]{\overrightarrow{#1}}

\makeatletter
\def\mr@ignsp#1 {\ifx\:#1\@empty\else #1\expandafter\mr@ignsp\fi}%
\newcommand{\multiref}[1]{\begingroup%\let\protect\string%
\xdef\mr@no@sparg{\expandafter\mr@ignsp#1 \: }%
\def\mr@comma{}%
\@for\mr@refs:=\mr@no@sparg\do{\mr@comma\def\mr@comma{,}\ref{\mr@refs}}%
\endgroup}
\makeatother

\newcommand{\hypref}[2]{\ifx\href\asklFrenkel:2013udaas #2\else\href{#1}{#2}\fi}
\newcommand{\Secref}[1]{Section~\multiref{#1}}

\tikzset{->-/.style={decoration={
  markings,
  mark=at position .5 with {\arrow{latex}}},postaction={decorate}}}
\tikzset{
    >=latex
    }

\newcommand{\wt}{\widetilde}
\newcommand{\mc}{\mathcal}
\newcommand{\nc}{\newcommand}
\nc{\on}{\operatorname}
\nc{\la}{\lambda}
\nc{\wh}{\widehat}
\nc{\ghat}{\wh\g}
\nc{\mb}{\mathbf}

\begin{document}
\title{$q$-opers, $QQ$-systems, and Bethe Ansatz}

\author[E. Frenkel]{Edward Frenkel}
\address{
Department of Mathematics,
University of California,
Berkeley, CA 94720, USA}

\author[P. Koroteev]{Peter Koroteev}
\address{
Department of Mathematics,
University of California,
Berkeley, CA 94720, USA}

\author[D.S. Sage]{Daniel S. Sage}
\address{
          Department of Mathematics, 
          Louisiana State University, 
          Baton Rouge, LA 70803, USA}

\author[A.M. Zeitlin]{Anton M. Zeitlin}
\address{
          Department of Mathematics, 
          Louisiana State University, 
          Baton Rouge, LA 70803, USA; 
          IPME RAS, St. Petersburg, Russia}

%\date{\today}

\numberwithin{equation}{section}

\begin{abstract}
  We introduce the notions of $(G,q)$-opers and Miura $(G,q)$-opers,
  where $G$ is a simply connected simple complex Lie group, and prove
  some general results about their structure. We then establish a
  one-to-one correspondence between the set of $(G,q)$-opers of a
  certain kind and the set of nondegenerate solutions of a system of
  Bethe Ansatz equations. This may be viewed as a $q$DE/IM
  correspondence between the spectra of a quantum integrable model
  (IM) and classical geometric objects ($q$-differential
  equations). If $\g$ is simply laced, the Bethe Ansatz equations we
  obtain coincide with the equations that appear in the quantum
  integrable model of XXZ-type associated to the quantum affine
  algebra $U_q \widehat\g$. However, if $\g$ is non-simply laced, then
  these equations correspond to a different integrable model, associated to
  $U_q {}^L\widehat\g$ where $^L\widehat\g$ is the Langlands dual
  (twisted) affine algebra. A key element in this $q$DE/IM
  correspondence is the 
  $QQ$-system that has appeared previously in the study of the ODE/IM
  correspondence and the Grothendieck ring of the category ${\mc O}$
  of the relevant quantum affine algebra.
\end{abstract}

\maketitle

\setcounter{tocdepth}{1}
\tableofcontents

%==============================================================================

\section{Introduction}

In his celebrated 1931 paper \cite{Bethe:1931hc}, Hans Bethe proposed
a method of diagonalization of the Hamiltonian of the XXX spin chain
model that was introduced three years earlier by Werner Heisenberg.
The elegance and simplicity of his method, dubbed the Bethe Ansatz,
has dazzled several generations of physicists and mathematicians.
Richard Feynman wrote \cite{doi:10.1142/9789814390187_0003}: ``I got
really fascinated by these (1+1)-dimensional models that are solved by
the Bethe ansatz and how mysteriously they jump out at you and work
and you don't know why. I am trying to understand all this better.''
In this paper, we make another attempt to understand all this better.

\subsection{Gaudin model}    \label{gaudin}

To explain our main idea, let us discuss a close relative of the XXX
spin chain: the quantum Gaudin model corresponding to a simple Lie
algebra $\g$. The space of states of this model is a representation of
the corresponding (polynomial) loop algebra $\g[t,t^{-1}]$. Suppose
for simplicity that it is the tensor product $\otimes_{i=1}^N
V_{\la_i}(z_i)$ of finite-dimensional evaluation representations with
highest weights $\la_i$ and distinct evaluations parameters $z_i$. In
this case, it was shown in \cite{Feigin:1994in,Frenkel:2004qy} that
the spectrum of the quantum Gaudin Hamiltonians can be encoded by
objects of an entirely different nature; namely, certain ordinary
differential operators on $\P^1$ called {\em opers}. (This concept was
introduced by Beilinson and Drinfeld \cite{Beilinson:2005} and goes
back to the work of Drinfeld and Sokolov on generalized KdV
systems \cite{Drinfeld:1985}). The opers that encode the spectrum have
regular singularities at the points $z_1,\ldots,z_N$ and $\infty$, and
trivial monodromy.

Remarkably, the opers that encode the spectra of the Gaudin model
associated to $\g$ are naturally associated to the {\em Langlands
  dual} Lie algebra $^L\g$ rather than $\g$ itself. This is no
accident; as explained in \cite{Frenkel:1995zp}, this correspondence
may be viewed as a special case of the construction \cite{BD:Hitchin}
of the geometric Langlands correspondence.

As a bonus, we obtain (at least, in a generic situation) explicit
formulas for the eigenvectors and eigenvalues of the Gaudin
Hamiltonians \cite{Feigin:1994in}, which are analogous to Bethe's
original formulas. An interesting aspect is that the $^L\g$-opers
corresponding to these eigenvalues can be expressed in terms of the
Miura transformation well-known in the theory of generalized KdV
equations \cite{Feigin:1994in,Frenkel:2004qy}. This means that we may
alternatively encode the eigenvalues by so-called Miura opers. The
same is true if we replace finite-dimensional representations
$V_{\la_i}$ by other highest weight $\g$-modules. (For more general
$\g$-modules, the spectral problem becomes more complicated; a
strategy for describing the relevant $^L\g$-opers and the
corresponding Bethe Ansatz equations using the data of 4D gauge
theories was proposed in \cite{Nekrasov:2010ka,Nekrasov:2011bc}.)

Accordingly, we obtain a link between two worlds that at first glance
seem to be far apart: the quantum world of integrable models and the
classical world of geometry and differential equations. The mystery is
partially resolved when we realize that this link can be deformed in
such a way that both sides correspond to quantum field theories.
Namely, the spectral problem of the Gaudin model deforms to the KZ
equations on conformal blocks in a WZW model associated to $\g$ (see
\cite{Reshetikhin:1994qw}), whereas opers deform to conformal blocks
of the ${\mc W}$-algebra corresponding to $^L\g$.

Thus, the original Gaudin/opers correspondence can be recovered in a
special limit of a more familiar duality of QFTs. In this limit, the
symmetry algebra on one side remains quantum but develops a large
commutative subalgebra (in fact, it can be obtained from the center of
$\wt{U}(\ghat)$ at the critical level
\cite{Feigin:1994in,Frenkel:2004qy} so that the quadratic Hamiltonians
correspond to the limit of the rescaled stress tensor given by the
Segal-Sugawara formula). In contrast, the limit of the symmetry
algebra on the other side is purely classical (opers and related
structures); thus, we get a commutative algebra with a Poisson
structure. This Poisson structure is a hint of the existence of
a one-parameter deformation, as is the appearance of the Miura
transformation (which is known to preserve Poisson structures) in the
formula for the eigenvalues of quantum Hamiltonians, when they are
expressed as opers.

\subsection{Quantum KdV systems}    \label{ODEIM}

The above construction has various generalizations. One possibility,
discussed in \cite{Zeitlin:2013iya}, is to replace the Lie algebra
$\g$ by a Lie superalgebra. Another possibility is to replace the
finite-dimensional simple Lie algebra $\g$ by the corresponding affine
Kac-Moody algebra $\ghat$. These affine Gaudin models were introduced
in \cite{FF:kdv}, where it was argued that the spectra of the
corresponding quantum Hamiltonians (in this case, there are both local
and non-local Hamiltonians) should be encoded by affine analogues of
$^L\g$-opers, called $^L\ghat$-opers.  Here, $^L\ghat$ is the affine
algebra that is Langlands dual to $\ghat$. In particular, $^L\ghat$ is
a twisted affine Kac-Moody algebra if $\ghat$ is an untwisted affine
Lie algebra associated to a non-simply laced $\g$, so this duality is
more subtle than in the finite-dimensional case.

It was shown in \cite{FF:kdv} that with particular choices of the
available parameters, the affine Gaudin model associated to $\ghat$
can be viewed as a quantization of the classical $\ghat$-KdV system
defined in \cite{Drinfeld:1985}. The conjecture of \cite{FF:kdv} (see
also \cite{Frenkel:ac}) states that the spectra of the corresponding
quantum $\ghat$-Hamiltonians can be encoded by $^L\ghat$-opers on
$\P^1$ with particular analytic properties. In the case of
$\ghat=\wh\sl_2$ (so that we are dealing with the KdV system proper),
these opers can be written as second order ordinary differential
operators with spectral parameter on $\P^1$. The conjecture of
\cite{FF:kdv} then becomes the conjecture made earlier by Bazhanov,
Lukyanov, and Zamolodchikov \cite{BLZ}, which was in fact a motivation
for \cite{FF:kdv}. (A similar conjecture was also made in the case of
$\ghat=\wh\sl_3$ in \cite{Bazhanov:2001xm}.)

For general $\g$, $^L\ghat$-opers are Lie algebra-valued ordinary
differential operators on $\P^1$ (or a finite cover, if $\g$ is
non-simply laced \cite{Frenkel:ac}) with spectral parameter.
Therefore the conjecture of \cite{FF:kdv,Frenkel:ac} fits the general
paradigm of the {\em ODE/IM correspondence} (see
e.g. \cite{Dorey:2007zx}), i.e. a correspondence between spectra of
quantum Hamiltonians in a quantum integrable model (IM) and ordinary
differential operators (ODE).

In \cite{Masoero_2016,Masoero_2016_SL}, Masoero, Raimondo, and Valeri
made an important discovery. They analyzed the $^L\ghat$-affine opers
that according to \cite{FF:kdv} should encode the eigenvalues on the
ground eigenstates of the $\ghat$-KdV system. They were able to assign
to each of them a collection of functions $\{ Q_i(z),\wt{Q}_i(z)
\}_{i=1,\ldots,r}$ solving a novel system of equations (dubbed the
$Q\wt{Q}$-system in \cite{Frenkel:ac}) generalizing the quantum
Wronskian relation found in \cite{BLZ}. Furthermore, they showed that
in a generic situation, these equations imply that the roots of the
functions $Q_i(z)$ satisfy a version of the system of Bethe Ansatz
equations proposed earlier in
\cite{OGIEVETSKY1986360,RW,Reshetikhin:1987}. These results were
generalized in \cite{Masoero:2018rel} to the affine opers
corresponding to the excited eigenstates of the $\ghat$-KdV system for
simply laced $\g$. Thus, we obtain that the Bethe Ansatz equations
describe objects on the ODE side of the ODE/IM correspondence. (This is
analogous to the interpretation of the Bethe Ansatz equations of the
$\g$-Gaudin model in terms of the corresponding $^L\g$-opers, see
Section \ref{gaudin}.)

To see that the same Bethe Ansatz equations also appear on the IM side,
we need to interpret the $Q\wt{Q}$-system directly in terms of the
quantum Hamiltonians. This was done in \cite{Frenkel:ac}, where it was
shown that the $Q\wt{Q}$-system arises from universal relations
between the quantum $\ghat$-KdV Hamiltonians corresponding to certain
specific transfer-matrices of $U_q\ghat$ (or rather its Borel
subalgebra).

Thus, we obtain an analogue of the Gaudin/oper correspondence of
Section \ref{gaudin}: the ODE/IM correspondence between the spectra of
the quantum Hamiltonians of the quantum KdV system and classical
geometric objects; namely, affine opers. The two sides of the ODE/IM
correspondence share the $Q\wt{Q}$-system and the Bethe Ansatz
equations (see \cite{Frenkel:ac} and Section \ref{prior} for more
details).

\subsection{Quantum spin chains}

The goal of the present paper is to express in the same spirit the
spectra of the quantum integrable models of the type considered by
Bethe in his original work \cite{Bethe:1931hc}.

In fact, it is known that the XXX spin chain naturally corresponds to the
Yangian of $\sl_2$, and this model can be generalized so that the
symmetry algebra is the Yangian of an arbitrary simple Lie algebra
$\g$. In this paper, we focus on the trigonometric versions of these
models. In the simplest case of $\g=\sl_2$, this is the Heisenberg XXZ
model, in which the symmetry algebra is the quantum affine algebra
$U_q \wh\sl_2$, where $q$ corresponds to the parameter $\Delta$ of the
XXZ model by the formula $\Delta=(q+q^{-1})/2$.  (The XXX model can be
obtained from the XXZ model, if we take the limit $q \rightarrow 1$ in
a particular fashion.) This model can be generalized to an arbitrary
$\g$, so that the symmetry algebra is the quantum affine algebra
$U_q\ghat$. We will refer to these models as {\em $U_q\ghat$ XXZ-type
  models}.

Analogues of the Bethe Ansatz equations for these models were proposed
in the 1980's \cite{OGIEVETSKY1986360,RW,Reshetikhin:1987}. In the
1990's, Reshetikhin and one of the authors \cite{Frenkel:ls} explained
how to relate them to the spectra of the transfer-matrices associated
to finite-dimensional representations of $U_q\ghat$, which are quantum
Hamiltonians of this model. This derivation was based on a conjectural
formula (generalizing Baxter's $TQ$-relation in the case of
$\g=\sl_2$) for the eigenvalues of these transfer-matrices in terms of
the $q$-characters of the corresponding representations of $U_q\ghat$.
This conjecture was proved by Hernandez and one of the authors in
\cite{Frenkel:2013uda}.

Thus, as in the previous two types of quantum integrable systems, we
obtain a description of the spectra of quantum Hamiltonians in terms
of solutions of a system of Bethe Ansatz equations.

The question then is: {\em How to construct the geometric objects
  encoding the solutions of these Bethe Ansatz equations?} If we
answer this question, we obtain another instance of the duality
discussed above.

The first step in this direction was made by Mukhin and Varchenko
\cite{Mukhin_2005} in the closely related case of the XXX-type models
corresponding to Yangians rather than quantum affine algebras. To each
solution of the corresponding system of Bethe Ansatz equations, they
assigned a difference operator that they called a ``discrete Miura
oper.'' However, since they did not give an independent definition of
such objects, this does not give a duality correspondence of the kind
discussed above.

In \cite{KSZ}, three of the authors proposed such a definition in the
case of the $U_q\wh\sl_n$ XXZ-type model. Namely, they introduced
$(\SL(n),q)$-opers (as scalar $q$-difference operators of order $n$)
and related them to nondegenerate solutions of the corresponding Bethe
Ansatz equations.  

In this paper, we elucidate the results of
\cite{KSZ} and generalize them to other $U_q\ghat$ XXZ-type models.

\subsection{$q$-opers and Miura $q$-opers}

We start in Section \ref{Sec:DefinitionsqOpers} with an intrinsic,
coor\-dinate-independent definition of the geometric objects dual to
the spectra of the $U_q\ghat$ XXZ-type models. These are the $q$-{\em
  opers} and {\em Miura $q$-opers} associated to an arbitrary simple
simply-connected complex Lie group $G$. (Our definition can be
generalized in a straightforward fashion to an arbitrary reductive
$G$.) They are analogues of ordinary opers and Miura opers.

Unlike opers (or more general connections on principle bundles), which
can be defined over an arbitrary complex algebraic curve $X$, the
definition of a $q$-connection (with $q$ not a root of unity) requires
that $X$ carry an automorphism of infinite order. In this paper, we
focus on the case $X=\P^1$, with the automorphism being $z \mapsto qz$.
However, a similar definition can also be given in the additive case
(a translation of the affine line, which can be extended to $\P^1$) or
the elliptic case, where $X$ is an elliptic curve and we choose a
translation by a generic element of the abelian group of its
points. Most of our results can be generalized to these two cases.

Having introduced the notion of a $q$-connection on a principal
$G$-bundle on the projective line $\P^1$, we define a $(G,q)$-oper as
a triple consisting of a principal $G$-bundle $\cF_G$ on $\P^1$
together with a reduction $\cF_{B_-}$ of $\cF_G$ to a Borel subgroup
$B_-$ and
a $q$-connection satisfying a certain condition with respect to
$\cF_{B_-}$, which is analogous to the oper condition. In defining the
$q$-oper condition, we follow the definition of $q$-Drinfeld-Sokolov
reduction given in \cite{Frenkel1998,1998CMaPh.192..631S}, which
involves the Bruhat cell of a Coxeter element of the Weyl group of
$G$.

We then define a Miura $(G,q)$-oper by analogy with the differential
case (see \cite{Frenkel:2003qx,Frenkel:2004qy}) as a $(G,q)$-oper with
an additional structure: a reduction $\cF_{B_+}$ of $\cF_G$ to an
opposite Borel subgroup $B_+$ that is preserved by the oper
$q$-connection. We prove a fundamental property of the two reductions
$\cF_{B_-}$ and $\cF_{B_+}$: they are necessarily in generic relative
position on a dense Zariski open subset of $\P^1$. This is analogous
to the situation in the differential case
\cite{Frenkel:2003qx,Frenkel:2004qy}.

\subsection{$q$DE/IM correspondence}

Consider now the $U_q\ghat$ XXZ-type model, where $\g$ is a simple Lie
algebra. Assume that the space of states on which the quantum
Hamiltonians (the transfer-matrices of $U_q\ghat$) act is
finite-dimensional. In the case of the Gaudin model, we saw in Section \ref{gaudin}
that the opers encoding the spectra in the finite-dimensional case
have trivial monodromy, i.e. are gauge equivalent to the trivial
connection. The XXZ-type models that we consider in this paper are
slightly more general, in that we include a non-trivial twist of the
boundary conditions, which is represented by a regular semisimple
element $Z$ of the Lie group $G$.  (In the case of the Gaudin model,
this corresponds to allowing opers on $\P^1$ with an irregular
singularity at $\infty$~\cite{Feigin:2006xs}; a similar twist can
also be included in XXX-type models.) Because of this twist, we
impose the condition that our $q$-opers are $q$-gauge equivalent to
the constant $q$-connection equal to $Z$. We call such $q$-opers, and
the corresponding Miura $q$-opers, $Z$-{\em twisted}.

We further introduce certain intermediate objects, which we call {\em
  $Z$-twisted Miura-Pl\"ucker $q$-opers} on $\P^1$. The main theorem
of this paper establishes a one-to-one correspondence between the set
of these objects satisfying an explicit nondegeneracy condition and
the set of nondegenerate solutions of a system of Bethe Ansatz
equations.

This is proved by constructing a bijection between each of the above
two sets and the set of nondegenerate solutions of a system of
algebraic equations. Remarkably, for simply laced $\g$, this system is
a version of the $Q\wt{Q}$-system discussed in Section \ref{ODEIM} in
the context of the quantum $\ghat$-KdV systems. Here, we refer to it
as the {\em $QQ$-system}. Following the results of
\cite{Frenkel:2013uda,Frenkel:ac}, we expect that this system of
equations is satisfied by a certain family of transfer-matrices of the
$U_q\ghat$ XXZ-type model (see Section \ref{prior} for more details).
Accordingly, we can link the spectra of these transfer-matrices and
$q$-opers.

Thus, for simply laced $\g$ we fulfill our goal and obtain a dual
description of the spectra of quantum Hamiltonians of the $U_q\ghat$
XXZ-type model in terms of geometric objects; namely, $Z$-twisted
Miura-Pl\"ucker $(G,q)$-opers on $\P^1$, where $Z$ corresponds to the
twist in the boundary condition of the XXZ-type model.

By analogy with the ODE/IM correspondence discussed in Section
\ref{ODEIM}, we call it the {\em $q$DE/IM correspondence}.

If $\g$ is non-simply laced, this $q$DE/IM correspondence becomes more
subtle. In fact, the $QQ$-system and the system of Bethe Ansatz
equations we obtain in this case are different from the systems that
arise from the $U_q\ghat$ XXZ-type model. (This was already noted in
\cite{Mukhin_2005} in the case of XXX-type models.) In light
of the results and conjectures of \cite{FR:w,Frenkel:ac}, we expect
that they arise from a novel quantum integrable model associated
to the {\em twisted} quantum affine algebra $U_q{}^L\ghat$, where
$^L\ghat$ is the affine Kac-Moody algebra Langlands dual to
$\ghat$. This will be explained in \cite{FrenkelHern:new}.

In summary, we expect that {\em the $q$DE/IM correspondence links
  $Z$-twisted Miura-Pl\"ucker $(G,q)$-opers on $\P^1$ and the spectra
  of quantum Hamiltonians in a quantum integrable model associated to
  $U_q{}^L\ghat$.}

\subsection{Quantum $q$-Langlands Correspondence}

In the case of the Gaudin model, both sides of the duality between
spectra and opers can be deformed to QFTs, as we discussed in Section
\ref{gaudin}. It turns out that such a deformation also exists in the
case of the $q$DE/IM correspondence. It was proposed in
\cite{Aganagic:2017smx} under the name ``quantum $q$-Langlands
correspondence.'' Specifically, it was shown in
\cite{Aganagic:2017smx} that for simply laced $\g$, one can identify
solutions of the $q$KZ equations, which can be viewed as deformed
conformal blocks for the quantum affine algebra, and deformed
conformal blocks for the deformed ${\mc W}$-algebra
$\mathcal{W}_{q,t}(\mathfrak{g})$. (The notation is a bit misleading
because what we previously denoted by $q$ is now the $t$ of
$\mathcal{W}_{q,t}(\mathfrak{g})$, and the $q$ of
$\mathcal{W}_{q,t}(\mathfrak{g})$ is the product of $t$ and the
parameter of the quantum affine algebra on the other side; thus, we
obtain an XXZ-type model in the limit in which this $q$ goes to $1$
and so we should consider $t$-opers rather than $q$-opers.)

As argued in \cite{Aganagic:2017smx}, this quantum duality has its
origins in the duality of the little string theory with defects on an
ADE-type surface. The defects wrap compact and noncompact cycles of
the internal Calabi-Yau manifold, thereby producing screening and
vertex operators respectively of the deformed conformal blocks of
$\mathcal{W}_{q,t}(\mathfrak{g})$.  On the other hand, little string
theory, in the limit where it becomes a conformal $(0,2)$ theory in
six dimensions, localizes on the defects, yielding quiver gauge
theories whose Higgs branches are Nakajima quiver varieties. Using
powerful methods of enumerative equivariant K-theory
\cite{Okounkov:2015aa}, the authors of \cite{Aganagic:2017smx} were
able to express these deformed conformal blocks in terms of solutions
of the $q$KZ equations.

The quantum $q$-Langlands correspondence for the QFTs associated to
non-simply laced Lie algebras is more subtle. In
\cite{Aganagic:2017smx}, it is studied by introducing an $H$-twist to
the little string theory which, together with the corresponding outer automorphism of the ADE
Dynkin diagram, acts on the complex line supporting the
defects. (A similar construction was considered in
\cite{Dey:2016qqp} and \cite{Kimura:2017hez}.)

In the critical level limit, solutions of the $q$KZ equations
corresponding to $U_q\ghat$ give rise to eigenvectors of the
$U_q\ghat$ XXZ-type model. The results of the present paper suggest
that for non-simply laced $\g$, the limits of the correlation
functions on the other side should give rise {\em not} to $q$-opers
associated to $^L G$, as one might expect by analogy with the duality
in the case of the Gaudin model, but rather to some twisted $q$-opers
associated to a twisted affine Kac-Moody group. These can probably be
defined similarly to the definition of twisted opers in \cite{MR2600880}
and twisted affine opers in \cite{Frenkel:ac}.

The appearance of the Langlands dual affine Lie algebra in this
duality can also be seen from the Poisson structure on the space of
$(G,q)$-opers which was defined in
\cite{Frenkel1998,1998CMaPh.192..631S} using $q$-Drinfeld-Sokolov
reduction. (This is, of course, a hint at the existence of the quantum
duality which is analogous to the classical duality involving the
Gaudin models, see Section \ref{gaudin}.)
The results and conjectures of \cite{FR:w} (see Conjecture 3, Section
6.3, and Appendix B) show that if $G$ is non-simply laced, then this
Poisson algebra (with $q$ replaced by $t^{r^\vee}$, $r^\vee$ being the
lacing number of $\g$) is isomorphic to the limit of ${\mc
  W}_{q,t}(\g)$ as $q \rightarrow 1$, which is in turn related to the
center of $U_t\wh{\g}^\vee$ at the critical level and to the algebra
of transfer-matrices of $U_t \wh{\g}^\vee$, where $\wh{\g}^\vee =
{}^L(\wh{^L\g})$. On the other hand, Conjecture 4 of \cite{FR:w}
states that the limit of ${\mc W}_{q,t}(\g)$ as $q \rightarrow e^{\pi
  i/r^\vee}$, is related to the center of $U_t{}^L\wh{\g}$ at the
critical level and to the algebra of transfer-matrices of $U_t
{}^L\wh{\g}$. Both of these limits and their potential connections to
the duality of \cite{Aganagic:2017smx} in the non-simply laced case
certainly deserve investigation. This will be further discussed in
\cite{FrenkelHern:new}.

\subsection{A brane construction}

Finally, we want to mention another approach to the $q$DE/IM
correspondence. In the case of the Gaudin model discussed in Section
\ref{gaudin}, it was shown in \cite{Nekrasov:2010ka,Gaiotto:2011nm}
that the spectra of the Gaudin Hamiltonians can be identified with the
intersection of the brane of opers and another brane in the
corresponding Hitchin moduli space of Higgs bundles. Recent work of
Elliott and Pestun \cite{Elliott:2018yqm} suggests that there is a
$q$-deformation of this construction, in which one should consider a
multiplicative version of Higgs bundles and a brane of $q$-opers. In
the case of finite-dimensional representations of quantum affine
algebras, this idea is supported by the fact, originally observed in
\cite{FR:w,Frenkel:ls} (see also Section
\ref{Sec:DSReduction} below), which is discussed in
\cite{Elliott:2018yqm}, that the $q$-character homomorphism and the
corresponding generalized Baxter $TQ$-relations can be interpreted in
terms of a $q$-deformation of the Miura transformation.

It is possible that this construction can also be generalized to
infinite-dimensional representations of quantum affine algebras along
the lines of \cite{Nekrasov:2011bc}, with the brane of opers replaced
by a brane of $q$-opers, which can now be rigorously defined using the
results of the present paper.

\subsection{Plan of the Paper}

The paper is organized as follows. In \Secref{Sec:DefinitionsqOpers},
we give the definition of meromorphic $q$-opers and Miura $q$-opers.
(We consider the case of the curve $\P^1$ but the same definition can
be given for any curve equipped with an automorphism of infinite
order.) We prove Theorem \ref{gen rel pos} about the relative position
of the two Borel reductions of a $q$-Miura oper. In Section
\ref{Sec:Ztw}, we define $Z$-twisted $(G,q)$-opers with regular
singularities on $\mathbb{P}^1$ as well as the corresponding Miura
$(G,q)$-opers and Cartan $q$-connections. In Section
\ref{Sec:nondegMiura}, we define Miura-Pl\"ucker $q$-opers and
introduce a nondegeneracy condition for them.

In \Secref{Sec:SL2review}, we consider in detail
the case $G=\SL(2)$ and show, elucidating the results of \cite{KSZ},
that the nondegenerate $Z$-twisted Miura $q$-opers (which are the same
as Miura-Pl\"ucker $q$-opers in this case) are in bijection with
nondegenerate solutions of the $QQ$-system and Bethe Ansatz equations
(Theorems \ref{isom sl2} and \ref{BASL2}). In \Secref{Sec:QQsystem},
we generalize these results to an arbitrary simply connected simple complex
Lie group $G$ (Theorems \ref{inj} and \ref{BAE}).

In \Secref{Sec:Backlund}, we define B\"acklund-type transformations on
Miura-Pl\"ucker $q$-opers, following the construction of similar
transformations in the Yangian case given in \cite{Mukhin_2005}. We
use these transformation to give a sufficient condition for a
Miura-Pl\"ucker $q$-oper to be a Miura $q$-oper. Finally, in
\Secref{Sec:DSReduction}, we discuss the relation between $(G,q)$-opers
and the $q$-Drinfeld-Sokolov reduction defined in
\cite{Frenkel1998,1998CMaPh.192..631S}. We construct a canonical
system of coordinates on the space of $(G,q)$-opers with regular
singularities. Following \cite{FR:w,Frenkel:ls}, we conjecture that if
$G$ is simply laced, then the formulas expressing these canonical
coordinates in terms of the polynomials $Q^i_+(z)$ coincide with the
generalized Baxter $TQ$-relations established in \cite{Frenkel:2013uda}.

\bigskip

\noindent{\bf Acknowledgments.} E.F. thanks D. Hernandez, N.
Nekrasov, and N. Reshetikhin for useful discussions. D.S.S. was
partially supported by NSF grant DMS-1503555 and a Simons
Collaboration Grant.  He also thanks Greg Muller for useful
discussions.  A.M.Z. is partially supported by
Simons Collaboration Grant, Award ID: 578501.

\section{$(G,q)$-opers with regular singularities}
\label{Sec:DefinitionsqOpers}

\subsection{Group-theoretic data}    \label{regsing} 

Let $G$ be a connected, simply connected, simple algebraic group of
rank $r$ over $\mathbb{C}$.  We fix a Borel subgroup $B_-$ with
unipotent radical $N_-=[B_-,B_-]$ and a maximal torus $H\subset B_-$.
Let $B_+$ be the opposite Borel subgroup containing $H$.  Let $\{
\alpha_1,\dots,\alpha_r \}$ be the set of positive simple roots for
the pair $H\subset B_+$.  Let $\{ \check\alpha_1,\dots,\check\alpha_r
\}$ be the corresponding coroots; the elements of the Cartan matrix of
the Lie algebra $\fg$ of $G$ are given by $a_{ij}=\langle
\alpha_j,\check{\alpha}_i\rangle$. The Lie algebra $\fg$ has Chevalley
generators $\{e_i, f_i, \check{\alpha}_i\}_{i=1, \dots, r}$, so that
$\fb_-=\Lie(B_-)$ is generated by the $f_i$'s and the
$\check{\alpha}_i$'s and $\fb_+=\Lie(B_+)$ is generated by the $e_i$'s
and the $\check{\alpha}_i$'s.  Let $\omega_1,\dots\omega_r$ be the
fundamental weights, defined by $\langle \omega_i,
\check{\alpha}_j\rangle=\delta_{ij}$.

Let $W_G=N(H)/H$ be the Weyl group of $G$. Let $w_i\in W$, $(i=1,
\dots, r)$ denote the simple reflection corresponding to
$\alpha_i$. We also denote by $w_0$ be the longest element of $W$, so
that $B_+=w_0(B_-)$.  Recall that a Coxeter element of $W$ is a
product of all simple reflections in a particular order. It is known
that the set of all Coxeter elements forms a single conjugacy class in
$W_G$. We will fix once and for all (unless otherwise specified) a
particular ordering $(\alpha_{i_1},\ldots,\alpha_{i_r})$ of the simple
roots. Let $c=w_{i_1}\dots w_{i_r}$ be the Coxeter element associated
to this ordering. In what follows (unless otherwise specified), all
products over $i \in \{ 1, \dots, r \}$ will be taken in this order;
thus, for example, we write $c=\prod_i w_i$.  We also fix
representatives $s_i\in N(H)$ of $w_i$. In particular, $s=\prod_i s_i$
will be a representative of $c$ in $N(H)$.

Although we have defined the Coxeter element $c$ using $H$ and $B_-$,
it is in fact the case that the Bruhat cell $BcB$ makes sense for any
Borel subgroup $B$.  Indeed, let $(\Phi,\Delta)$ be the root system
associated to $G$, where $\Delta$ is the set of simple roots as
above and $\Phi$ is the set of all roots.  These data give a
realization of the Weyl group of $G$ as a Coxeter group, i.e., a pair
$(W_G,S)$, where $S$ is the set of Coxeter generators $w_i$ of $W_G$
associated to elements of $\Delta$. Now, given any Borel subgroup $B$,
set $\fb = \on{Lie}(B)$. Then the dual of the vector space
$\fb/[\fb,\fb]$ comes equipped with a set of roots and
simple roots, and this pair is canonically isomorphic to the root
system $(\Phi,\Delta)$~\cite[\S 3.1.22]{CG}.  The definition of the
sets of roots and simple roots on this space involves a choice of
maximal torus $T\subset B$, but these sets turn out to be independent
of the choice. Accordingly, the group $N(T)/T$ together with the set
of its Coxeter generators corresponding to these simple roots is
isomorphic to $(W_G,S)$ as a Coxeter group. Under this isomorphism,
$w\in W_G$ corresponds to an element of $N(T)/T$ by the following
rule: we write $w$ as a word in the Coxeter generators of $W_G$
corresponding to elements of $S$ and then replace each Coxeter
generator in it by the corresponding Coxeter generator of
$N(T)/T$. Accordingly, the Bruhat cell $BwB$ is well-defined for any
$w\in W_G$.

\subsection{Meromorphic $q$-opers}

The definitions given below can be given for an arbitrary algebraic
curve equipped with an automorphism of infinite order. For the sake of
definitiveness, we will focus here on the case of the curve $\P^1$ and
its automorphism $M_q: \P^1 \to \P^1$ sending $z \mapsto qz$, where
$q\in\C^\times$ is {\em not} a root of unity.

Given a principal $G$-bundle $\cF_G$ over $\P^1$ (in the Zariski
topology), let $\cF_G^q$ denote its pullback under the map $M_q: \P^1
\to \P^1$ sending $z\mapsto qz$. A meromorphic $(G,q)$-{\em
  connection} on a principal $G$-bundle $\cF_G$ on $\P^1$ is a section
$A$ of $\Hom_{\cO_{U}}(\cF_G,\cF_G^q)$, where $U$ is a Zariski open
dense subset of $\P^1$. We can always choose $U$ so that the
restriction $\cF_G|_U$ of $\cF_G$ to $U$ is isomorphic to the trivial
$G$-bundle. Choosing such an isomorphism, i.e. a trivialization of
$\cF_G|_U$, we also obtain a trivialization of
$\cF_G|_{M_q^{-1}(U)}$. Using these trivializations, the restriction
of $A$ to the Zariski open dense subset $U \cap M_q^{-1}(U)$ can be
written as a section of the trivial $G$-bundle on $U \cap M_q^{-1}(U)$,
and hence as an element $A(z)$ of $G(z)$.\footnote{Throughout the
  paper, if $K$ is a complex algebraic group, we set $K(z)=K(\C(z))$.}
Changing the trivialization of $\cF_G|_U$ via $g(z) \in G(z)$ changes
$A(z)$ by the following $q$-{\em gauge transformation}:
\begin{equation}    \label{gauge tr}
A(z)\mapsto g(qz)A(z)g(z)^{-1}.
\end{equation}
This shows that the set of equivalence classes of pairs $(\cF_G,A)$ as
above is in bijection with the quotient of $G(z)$ by the $q$-gauge
transformations \eqref{gauge tr}.

Following \cite{Frenkel1998,1998CMaPh.192..631S}, we define a $(G,q)$-oper as a
$(G,q)$-connection on a $G$-bundle on $\P^1$ equipped with a reduction
to the Borel subgroup $B_-$ that is not preserved by the
$(G,q)$-connection but instead satisfies a special ``transversality
condition'' which is defined in terms of the Bruhat cell associated to
the Coxeter element $c$. Here is the precise definition.

\begin{Def}    \label{qop}
  A meromorphic $(G,q)$-{\em oper} (or simply a $q$-{\em oper}) on
  $\mathbb{P}^1$ is a triple $(\cF_G,A,\cF_{B_-})$, where $A$ is a
  meromorphic $(G,q)$-connection on a $G$-bundle $\cF_G$ on
  $\mathbb{P}^1$ and $\mathcal{F}_{B_-}$ is a reduction of $\cF_G$
  to $B_-$ satisfying the following condition: there exists a Zariski
  open dense subset $U \subset \P^1$ together with a trivialization
  $\imath_{B_-}$ of $\mathcal{F}_{B_-}$ such that the restriction of
  the connection $A: \cF_G \to \cF_G^q$ to $U \cap M_q^{-1}(U)$,
  written as an element of $G(z)$ using the trivializations of
  ${\mathcal F}_G$ and $\cF_G^q$ on $U \cap M_q^{-1}(U)$ induced by
  $\imath_{B_-}$, takes values in the Bruhat cell $B_-(\C[U \cap
  M_q^{-1}(U)]) c B_-(\C[U \cap M_q^{-1}(U)])$.
\end{Def}

Note that this property does not depend on the choice of
trivialization $\imath_{B_-}$.

Since $G$ is assumed to be simply connected, any $q$-oper connection
$A$ can be written (using a particular trivialization $\imath_{B_-}$)
in the form
\begin{equation}    \label{qop1}
A(z)=n'(z)\prod_i (\phi_i(z)^{\check{\alpha}_i} \, s_i )n(z),
\end{equation}
where $\phi_i(z) \in\C(z)$ and  $n(z), n'(z)\in N_-(z)$ are such that
their zeros and poles are outside the subset $U \cap M_q^{-1}(U)$ of
$\P^1$.

We remark that the choice of a particular Coxeter element $c$ in this
definition can be viewed as a choice of a particular gauge, at least
for orderings differing by a cyclic permutation. Indeed, we will
see below in Proposition~\ref{prop:coxchoice} that the spaces of
$q$-opers we consider for such a pair of Coxeter elements are
isomorphic under a specific $q$-gauge transformation.

\subsection{Miura $q$-opers}

We will also need a $q$-difference version of the notion of
differential Miura opers introduced in
\cite{Frenkel:2003qx,Frenkel:2004qy}. These are $q$-opers together
with an additional datum: a reduction of the underlying $G$-bundle to
the Borel subgroup $B_+$ (opposite to $B_-$) that is preserved by the
oper $q$-connection.

\begin{Def}    \label{Miura}
  A {\em Miura $(G,q)$-oper} on $\mathbb{P}^1$ is a quadruple
  $(\cF_G,A,\cF_{B_-},\cF_{B_+})$, where $(\cF_G,A,\cF_{B_-})$ is a
  meromorphic $(G,q)$-oper on $\P^1$ and $\cF_{B_+}$ is a reduction of
  the $G$-bundle $\cF_G$ to $B_+$ that is preserved by the
  $q$-connection $A$.
\end{Def}

Forgetting $\cF_{B_+}$, we associate a $(G,q)$-oper to a given Miura
$(G,q)$-oper. We will refer to it as the $(G,q)$-oper underlying the
Miura $(G,q)$-oper.

The following result is an analogue of a statement about differential
Miura opers proved in \cite{Frenkel:2003qx,Frenkel:2004qy}.

Suppose we are given a principal $G$-bundle $\cF_G$ on any smooth
complex manifold $X$ equipped with reductions $\cF_{B_-}$ and
$\cF_{B_+}$ to $B_-$ and $B_+$ respectively. We then assign to any
point $x \in X$ an element of the Weyl group $W_G$. To see this, first
note that the fiber
$\cF_{G,x}$ of $\cF_G$ at $x$ is a $G$-torsor with reductions
$\cF_{B_-,x}$ and $\cF_{B_+,x}$ to $B_-$ and $B_+$
respectively. Choose any trivialization of $\cF_{G,x}$, i.e. an
isomorphism of $G$-torsors $\cF_{G,x} \simeq G$. Under this
isomorphism, $\cF_{B_-,x}$ gets identified with $aB_- \subset G$ and
$\cF_{B_+,x}$ with $bB_+$. Then, $a^{-1}b$ is a well-defined element of
the double quotient $B_-\backslash G/B_+$, which is in bijection with
$W_G$. Hence, we obtain a well-defined element of $W_G$.

We will say that $\cF_{B_-}$ and $\cF_{B_+}$ have  {\em generic
  relative position} at $x \in X$ if the element of $W_G$ assigned to
them at $x$ is equal to $1$. This means that the corresponding element
$a^{-1}b$ belongs to the open dense Bruhat cell $B_-B_+ \subset
G$.

\begin{Thm}    \label{gen rel pos}
  For any Miura $(G,q)$-oper on $\mathbb{P}^1$, there exists an open
  dense subset $V \subset \P^1$ such that the reductions $\cF_{B_-}$
  and $\cF_{B_+}$ are in generic relative position for all $x \in V$.
\end{Thm}

\begin{proof}
  Let $U$ be a Zariski open dense subset on $\P^1$ as in Definition
  \ref{qop}. Choosing a trivialization $\imath_{B_-}$ of $\cF_G$ on $U
  \cap M_q^{-1}(U)$, we can write the $q$-connection $A$ in the form
  \eqref{qop1}. On the other hand, using the $B_+$-reduction
  $\cF_{B_+}$, we can choose another trivialization of $\cF_G$ on $U
  \cap M_q^{-1}(U)$ such that the $q$-connection $A$ acquires the form
  $\wt{A}(z) \in B_+(z)$. Hence there exists $g(z) \in G(z)$ such that
\begin{equation}    \label{connecting}
g(zq) n'(z)\prod_i (\phi_i(z)^{\check{\alpha}_i} \, s_i )n(z)
g(z)^{-1} = \wt{A}(z) \in B_+(z).
\end{equation}
Recall the Bruhat decomposition (see \cite{Borel_1991}[Theorem
21.15]):
\begin{equation}    \label{Bruhat}
G(z) = \bigsqcup_{w \in W_G} B_+(z) w N_-(z).
\end{equation}
The statement of the proposition is equivalent to the statement that
$$
g(z) \in B_+(z) N_-(z)
$$
(corresponding to $w=1$), or equivalently, that $g(z)
\notin \; B_+(z) w N_-(z)$ for $w \neq 1$.

Suppose that this is not the case. Then $g(z) = b_+(z) w n_-(z)$ for
some $b_+(z) \in B_+(z), n_-(z) \in N_-(z)$, and $w \neq 1$. Setting
$\wt{n'}(z) = n_-(zq) n'(z)$ and $\wt{n}(z) = n(z) n_-(z)^{-1}$, we
can rewrite \eqref{connecting} as
\begin{equation}    \label{elt}
\wt{n'}(z) \prod_i (\phi_i(z)^{\check{\alpha}_i} \, s_i ) \wt{n}(z)
\in w B_+(z) w^{-1}.
\end{equation}

Now, the Borel subgroup decomposes as
$$
wB_{+}w^{-1}=H(N_-\cap wN_+ w^{-1})(N_+\cap
wN_+ w^{-1})
$$
because $wN_+w^{-1}=(N_-\cap wN_+ w^{-1})(N_+\cap wN_+
w^{-1})).$ Hence, denoting the element \eqref{elt} by $A$, we can
write
$$
A = h u_- u_+, \qquad h \in H, \quad u_- \in N_-\cap wN_+ w^{-1},
\quad u_+ \in N_+\cap wN_+ w^{-1}.
$$
It follows that $u_+\in B_-cB_-\cap N_+ \cap wN_+w^{-1}$. In
particular, $u_+ \in N_+ \cap wN^+w^{-1}$, which is the product of the
one-dimensional unipotent subgroups $X_\alpha$, where $\alpha$ runs
over the set of positive roots for which $w(\alpha)$ is positive.

On the other hand, according to Theorem \ref{gen elt}, every element
of $N_-\prod_i\phi_i(z)^{\check{\alpha}_i}s_iN_-\cap B_+$ can be
written in the form
$$
\prod_i g_i(z)^{\check{\alpha}_i}e^{\frac{\phi_i(z) t_i}{g_i(z)}e_i}, \qquad
g_i(z) \in \C(z)^\times, t_i \in \C^\times.
$$
Therefore, $u_+ = h'(z) \prod_i e^{a_i(z) e_i}$, where $h\in H$ and
$a_i(z) \neq 0$ for all $i=1,\ldots,r$. Such an element $u_+$ can
belong to $wN_+w^{-1}$ only if $w^{-1}$ maps all positive simple roots
to positive roots, i.e. if it preserves the set  of positive
roots. But this can only happen for $w=1$. This completes the proof.
\end{proof}

\begin{Cor}    \label{gen rel pos1}
For any Miura $(G,q)$-oper on $\mathbb{P}^1$, there exists a
trivialization of the underlying $G$-bundle $\cF_G$ on an open
dense subset of $\P^1$ for which the oper $q$-connection has the form
\begin{equation}    \label{genmiura}
A(z)\in N_-(z)\prod_i((\phi_i(z)^{\check{\alpha}_i}s_i
)N_-(z) \; \cap \; B_+(z).
\end{equation}
\end{Cor}

\begin{proof}
In the course of the proof of Theorem \ref{gen rel pos}, we showed
that we can choose a trivialization of $\cF_G$ so that the oper
$q$-connection has the form
$$
\wt{A}(z) = g(zq) n'(z)\prod_i (\phi_i(z)^{\check{\alpha}_i} \, s_i )n(z)
g(z)^{-1},
$$
where $n(z), n'(z) \in N_-(z)$ and $g(z) = b_+(z)n_-(z)$, with $b_+(z)
\in B_+(z), n_-(z) \in N_-(z)$. Therefore, changing the trivialization
by $b_+(z)$, we obtain the $q$-connection
$$
A(z) = b_+(zq)^{-1} \wt{A}(z) b_+(z) \in
N_-(z)\prod_i((\phi_i(z)^{\check{\alpha}_i}s_i
)N_-(z) \; \cap \; B_+(z).
$$
\end{proof}

\subsection{Explicit representatives}

Our proof of Theorem \ref{gen rel pos} relies on the following general
result, which might be of independent interest.

For a field $F$, consider the group $G(F)$ and the corresponding
subgroups $N_-(F)$, $B_+(F)$, and $H(F)$. As before, we denote by
$s_i$ a lifting of $w_i \in W_G$ to $G(F)$.

\begin{Thm}    \label{gen elt}
Let $F$ be any field, and fix $\lambda_i \in F^\times,
i=1,\ldots,r$. Then every element of the set
$N_-\prod_i\lambda_i^{\check{\alpha}_i}s_iN_- \; \cap \; B_+$ can be written
in the form
\begin{equation}    \label{gicheck}
\prod_i g_i^{\check{\alpha}_i}e^{\frac{\lambda_i t_i}{g_i}e_i}, \qquad
g_i \in F^\times,
\end{equation}
where each $t_i \in F^\times$ is determined by the lifting $s_i$.
\end{Thm}

We start with the following

\begin{Lem}    \label{two forms}
  Every element of $\lambda_i^{\check{\alpha}_i}s_iN_-$ may be written
  in either of the following two forms:
$$n_- \lambda_i^{\check{\alpha}_i}s_i \qquad \text{or} \qquad
n_- g^{\check{\alpha}_i}e^{\frac{\lambda_it_i}{g}e_i}$$
for some $n_-\in N_-(F)$, $g \in F^\times$, and with each $t_i \in
F^\times$ determined by the lifting $s_i$.
\end{Lem}

\begin{proof}
First, note that $\lambda_i^{\check{\alpha}_i}s_ie^{af_i}$ with $a\neq
0$ is of the form
$n_- g^{\check{\alpha}_i} e^{\frac{\lambda_it_i}{g}e_i}$, where $n_-\in
N_-$ and $g=a\lambda_it_i$.
This follows from the equality of $2\times2$ matrices
$$
\begin{pmatrix}
    \lambda_i &0\\
  0  &\lambda_i^{-1}
 \end{pmatrix} 
 \begin{pmatrix}
    0 &t_i\\
  -t_i^{-1}  &0
 \end{pmatrix}
\begin{pmatrix}
  1   &0\\  
  a  &1
 \end{pmatrix} =
 \begin{pmatrix}
  1   &0\\
  n_-  &1
 \end{pmatrix}
 \begin{pmatrix}
    a\lambda_it_i &\lambda_it_i\\
  0  &(a\lambda_it_i)^{-1}
 \end{pmatrix} \,,
 $$
where $n_-=-\frac{1}{a t_i^2\lambda_i^2}$.

An arbitrary element $u$ of $N_-$ can be expressed as a product
$$u=\prod_k e^{a_kf_k} \prod_{s<r}e^{a_{s,r}[f_s,f_r]}\dots,$$
where the ellipses stand for exponentials of higher commutators and we
assume a particular order in the first two products.  Notice that
if $a_i\neq 0$, then
$$\lambda^{\check{\alpha}_i}_is_i\prod_k
e^{a_kf_k}=n^1_-g_i^{\check{\alpha}_i}e^{\frac{\lambda_it_i}{g_i}e_i},$$
for some $n^1_-\in N_-$ and where $g_i=a_i t_i\lambda_i$. On the other
hand, if $a_i=0$, then
$$\lambda^{\check{\alpha}_i}_is_i\prod_k
e^{a_kf_k}=n^2_-\lambda_i^{\check{\alpha}_i}s_i$$ 
for some $n^2_-\in N_-$ since $s_i(f_k)$, $k\neq i$, belong to the Lie
algebra of $N_-$.
Therefore, we have 
\begin{equation}\label{eq:firststep}
\lambda^{\check{\alpha}_i}_is_in_ -=
\begin{cases}
{n}^1_-
g_i^{\check{\alpha}_i}e^{\frac{\lambda_it_i}{g_i}e_i}\prod_{s<r}e^{a_{s,r}[f_s,f_r]}\dots&
\text{if } a_i\neq 0,\\
 {n}^2_-\lambda_i^{\check{\alpha}_i}s_i\prod_{s<r}e^{a_{s,r}[f_s,f_r]}\dots&
 \text{if } a_i= 0.
\end{cases}
\end{equation}
Denote one of the products of commutators and higher commutators in \eqref{eq:firststep} by
  $X$. Clearly, $X \in \text{exp}([{\mathfrak n}_-,{\mathfrak n}_-])$,
  and therefore $s_i(X)$ and $e^{b_ie_i}Xe^{-b_ie_i}$ 
  belong to $\text{exp}({\mathfrak n}_-)$. This
  allows us to move the elements
  $g_i^{\check{\alpha}_i}e^{\frac{\lambda_it_i}{g_i}e_i}$,
  $\lambda_i^{\check{\alpha}_i}s_i$ to the right of the products in
  \eqref{eq:firststep} at the expense of multiplying $n^{1}_-, n^{2}_-$
  by additional elements from $N_-$. This completes the proof of the
  lemma.
\end{proof}

The following proposition is proved by repeated
applications of this lemma. Suppose that $s=s_{i_1} \ldots
s_{i_r}$. Below, the product over $i$ means the ordered product
corresponding to this decomposition.

\begin{Prop}    \label{ordered}
For every element $A$ of
  $N_-\prod_i\lambda_i^{\check{\alpha}_i}s_iN_-$, there exists a
  particular ordered subset $J = \{j_1, ..., j_k\}\subset \{i_1, \dots,
  i_s\}$ such that $A$ can be written as $n_- \prod_i \epsilon_i$,
where $n_- \in N_-$ and
$\epsilon_i=g_i^{\check{\alpha}_i}e^{\frac{\lambda_it_i}{g_i}e_i}$ for
$i \in J$ and $\lambda_i^{\check{\alpha}_i}s_i$ for $i \notin J$,
with $g_i, t_i \in \C^\times$.
\end{Prop}

\medskip

We are now ready to prove Theorem \ref{gen elt}.

\medskip

\noindent{\em Proof of Theorem \ref{gen elt}.} In the notation of
Proposition \ref{ordered}, $n_- \prod_i \epsilon_i$ belongs to $N_-wB_+$,
where $w=s_{j_1}\dots s_{j_k}$. Such an element can only belong to
$B_+$ if $w=1$, i.e. if the subset $J = \{j_1,\dots, j_k\}$ is
empty. This proves the theorem.\qed

\subsection{Connection to Fomin-Zelevinsky factorization}

Theorem \ref{gen elt} is closely related to a statement about the
Fomin-Zelevinsky factorization of a particular double Bruhat cell
(defined over $\C$) \cite{FZ}.  Here, we will only recall this
factorization for double Bruhat cells of the form $G^w=B_+\cap B_- w
B_-$, where $w\in W$. If $w=w_{i_1}\dots w_{i_k}$ is a reduced
expression, Fomin and Zelevinsky have shown that the map
\begin{align} \notag
H\times
\C^k &\rightarrow G \\    \label{FZmap}
(a,u_1,\dots,u_k) &\mapsto a e^{u_1 e_{i_1}}\dots
e^{u_k e_{i_k}}
\end{align}
induces an isomorphism of algebraic varieties $\psi^w:H\times
(\C^\times)^k\cong G^w_0$, where $G^w_0$ is a certain Zariski-open and
dense subset of $G^w$~\cite{FZ}. In particular, $G^w$ is irreducible,
and the dimension of $G^w$ is $r+\ell(w)$.

If one fixes a lifting $\wt{w}$ of $w$, one can also consider the
subvariety $C^{\wt{w}}=B_+\cap N_- \wt{w} N_-$.  It is easy to
see that $G^w=H C^{\wt{w}}$, so $C^{\wt{w}}$ is irreducible and
has dimension $\ell(w)$.  Moreover, $C^{\wt{w}}_0=C^{\wt{w}}\cap
G^w_0$ is a Zariski-open and dense subset of $C^{\wt{w}}$, and any
element of $C^{\wt{w}}_0$ can be expressed uniquely in terms of the
Fomin-Zelevinsky map \eqref{FZmap}.

We now specialize to the case of the Coxeter element $c$. The
factorization in Theorem \ref{gen elt} yields an explicit version of
the Fomin-Zelevinsky factorization for $C^{\wt{c}}_0$, where
$\wt{c}=\prod \lambda_i^{\check{\alpha}_i}s_i$.  Theorem \ref{gen
  elt} thus implies that $C^{\wt{c}}_0=C^{\wt{c}}$ and
$G^c_0=G^c$, i.e., in this case, the Fomin-Zelevinsky map \eqref{FZmap}
gives a factorization for the {\em entire} double Bruhat cell.  In
fact, the same argument applies to show $G^w_0=G^w$ for any $w$ whose
reduced decompositions do not involve repeated simple reflections.  We
remark that this statement is apparently known to specialists and may
also be proved using cluster algebra
techniques.\footnote{We thank Greg Muller for a discussion of these
  matters.}

\subsection{$q$-opers and Miura $q$-opers with regular singularities}

Let $\{ \Lambda_i(z) \}_{i=1,\ldots,r}$ be a collection of
nonconstant polynomials.

\begin{Def}    \label{d:regsing}
  A $(G,q)$-{\em oper with regular singularities determined by $\{
    \Lambda_i(z) \}_{i=1,\ldots,r}$} is a $q$-oper on $\P^1$ whose
  $q$-connection \eqref{qop1} may be written in the form
\begin{equation}    \label{Lambda}
A(z)= n'(z)\prod_i(\Lambda_i(z)^{\check{\alpha}_i} \, s_i)n(z), \qquad
n(z), n'(z)\in N_-(z).
\end{equation}
\end{Def}

\begin{Def}    \label{MiuraRS}
  {\em A Miura $(G,q)$-oper with regular singularities determined by
polynomials $\{ \Lambda_i(z) \}_{i=1,\ldots,r}$} is a Miura
  $(G,q)$-oper such that the underlying $q$-oper has
regular singularities determined by $\{ \Lambda_i(z)
\}_{i=1,\ldots,r}$.
\end{Def}

According to Corollary \ref{gen rel pos1}, we can write the
$q$-connection underlying such a Miura $(G,q)$-oper in the form
$$
A(z)\in N_-(z)\prod_i((\Lambda_i(z)^{\check{\alpha}_i}s_i
)N_-(z) \; \cap \; B_+(z).
$$

Recall Theorem \ref{gen elt}. Observe that we can choose liftings
$s_i$ of the simple reflections $w_i \in W_G$ in such a way that
$t_i=1$ for all $i=1,\ldots,r$. From now on, we will only consider
such liftings.

The following theorem follows from Theorem \ref{gen elt} in the case
$F=\C(z)$ and gives an explicit parametrization of generic elements
of the above intersection.

\begin{Thm}    \label{gen elt1}
Every element of
$N_-(z)\prod_i(\Lambda_i(z))^{\check{\alpha}_i}s_i)N_-(z) \; \cap \;
B_+$ may be written in the form
\begin{equation}    \label{form of A}
A(z)=\prod_i
g_i(z)^{\check{\alpha}_i} \; e^{\frac{\Lambda_i(z)}{g_i(z)}e_i}, \qquad
g_i(z) \in \C(z)^\times.
\end{equation}
\end{Thm}

\begin{Cor}    \label{Miura form}
For every Miura $(G,q)$-oper with regular singularities determined by
the polynomials $\{ \Lambda_i(z) \}_{i=1,\ldots,r}$, the underlying
$q$-connection can be written in the form \eqref{form of A}.
\end{Cor}

\section{$Z$-twisted $q$-opers and Miura $q$-opers}    \label{Sec:Ztw}

Next, we consider a class of (Miura) $q$-opers that are gauge
equivalent to a constant element of $G$ (as $(G,q)$-connections).  Let
$Z$ be an element of the maximal torus $H$. Since $G$ is
simply connected, we can write
\begin{equation}    \label{Z}
Z = \prod_{i=1}^r \zeta_i^{\check\alpha_i}, \qquad \zeta_i \in
\C^\times.
\end{equation}

\begin{Def}    \label{Ztwoper}
  A {\em $Z$-twisted $(G,q)$-oper} on $\mathbb{P}^1$ is a $(G,q)$-oper
  that is equivalent to the constant element $Z \in H \subset H(z)$
  under the $q$-gauge action of $G(z)$, i.e. if $A(z)$ is the
  meromorphic oper $q$-connection (with respect to a particular
  trivialization of the underlying bundle), there exists $g(z) \in
  G(z)$ such that
\begin{eqnarray}    \label{Ag}
A(z)=g(qz)Z g(z)^{-1}.
\end{eqnarray}
\end{Def}

\begin{Def}    \label{ZtwMiura}
A {\em $Z$-twisted Miura $(G,q)$-oper} is a Miura $(G,q)$-oper on
$\mathbb{P}^1$ that is equivalent to the constant element $Z \in H
\subset H(z)$ under the $q$-gauge action of $B_+(z)$, i.e.
\begin{eqnarray}    \label{gaugeA}
A(z)=v(qz)Z v(z)^{-1}, \qquad v(z) \in B_+(z).
\end{eqnarray}
\end{Def}

\subsection{From $Z$-twisted $q$-opers to Miura $q$-opers}

It follows from Definition \ref{Ztwoper} that any $Z$-twisted
$(G,q)$-oper is also $Z'$-twisted for any $Z'$ in the $W_G$-orbit of
$Z$. However, if we endow it with the structure of a $Z$-twisted Miura
$(G,q)$-oper (by adding a $B_+$-reduction $\cF_{B_+}$ preserved by the
oper $q$-connection), then we fix a specific element in this
$W_G$-orbit.

Indeed, suppose that $(\cF_G,A,\cF_{B_-})$ is a $Z$-twisted
$(G,q)$-oper. Choose a trivialization of $\cF_G$ on a Zariski open
dense subset $U$ of $\P^1$ with respect to which $A$ is equal to
$Z$. Then a choice of an $A$-invariant $B_+$-reduction $\cF_{B_+}$ of
$\cF_G$ on a Zariski open dense subset $V \subset U$ is the same as a
choice of a $Z$-invariant $B_+$-reduction of the fiber $\cF_{G,v}$ of
$\cF_G$ at any point $v \in V$. Our trivialization of $\cF_G|_U$
identifies $\cF_{G,v}$ with $G$, and hence a $B_+$-reduction of
$\cF_{G,v}$ with a right coset $gB_+$ of $G$. The $Z$-invariance of
this $B_+$-reduction means that $gB_+$, viewed as a point of the flag
variety $G/B_+$, is a fixed point of $Z$. This is equivalent to $Z \in
gB_+g^{-1}$ or $g^{-1}Zg \in B_+$.

Adding the $B_+$-reduction $\cF_{B_+}$ corresponding to a coset $gB_+$
satisfying this property to our $(G,q)$-oper $(\cF_G,A,\cF_{B_-})$, we
endow it with the structure of a Miura $(G,q)$-oper. A choice of
trivialization of $\cF_{B_+}$ is equivalent to a choice of an
identification of the coset $gB_+$ with $B_+$, which is the same as a
choice of an element of this coset (this element corresponds to $1 \in
B_+$ under the given isomorphism $B_+ \simeq gB_+$). Without
loss of generality, we denote this element also by $g$. Then, with
respect to the corresponding trivialization of the $(G,q)$-oper bundle
$\cF_G$, the $q$-connection becomes equal to $g^{-1}Zg \in
B_+$. However, note that because we can multiply $g$ on the right by
any element of $B_+$, we still have the freedom to conjugate
$g^{-1}Zg$ by an element of $B_+$, and there is a unique element in
the $B_+$-conjugacy class of $g^{-1}Zg$ of the form $w^{-1}Zw$, where
$w \in W_G$. Denote this element by $Z'$. We now conclude that the
Miura $(G,q)$-oper obtained by endowing our $(G,q)$-oper with the
$B_+$-reduction $\cF_{B_+}$ corresponding to $gB_+$ is $Z'$-twisted.

As a result, we also construct a map $\mu_Z$ from $(G/B_+)^Z = \{ f
\in G/B_+ \; | \; Z \cdot f = f \}$ to $W_G \cdot Z$, sending $gB_+$
with $g^{-1}Zg \in B_+$ to the unique element $Z'$ of $W_G \cdot Z$
that is $B_+$-conjugate to $g^{-1}Zg$. According to the above
construction, the set of points of the fiber $\mu_Z^{-1}(Z')$ of
$\mu_Z$ over a specific $Z' \in W_G \cdot Z$ is in bijection with the
set of $A$-invariant $B_+$-reductions $\cF_{B_+}$ on our $(G,q)$-oper
such that the corresponding Miura $(G,q)$-oper is $Z'$-twisted.

Thus, we have proved the following result.

\begin{Prop}    \label{Z prime}
  Let $Z \in H$. For any $Z$-twisted $(G,q)$-oper $(\cF_G,A,\cF_{B_-})$
  and any choice of $B_+$-reduction $\cF_{B_+}$ of $\cF_G$ preserved
  by the oper $q$-connection $A$, the resulting Miura $(G,q)$-oper is
  $Z'$-twisted for a particular $Z' \in W_G \cdot Z$.

  Moreover, the set of $A$-invariant $B_+$-reductions $\cF_{B_+}$ on
  the $(G,q)$-oper $(\cF_G,A,\cF_{B_-})$ making it into a $Z'$-twisted
  Miura $(G,q)$-oper is in bijection with the set of points of
  $\mu_Z^{-1}(Z')$.
\end{Prop}

Consider two extreme examples of $Z$. If $Z=1$, we always obtain
$Z'=1$, so the set of $B_+$-reductions is the set of points of the
flag manifold $G/B_+$.  On the other hand, if $Z\in H$ is regular,
then the map $\mu_Z$ is a bijection. Thus, for each $Z' \in W_G \cdot
Z$, there is a unique $B_+$-reduction on the $(G,q)$-oper
$(\cF_G,A,\cF_{B_-})$ making it into a $Z'$-twisted Miura
$(G,q)$-oper.  We will focus on the regular case in the rest of this
paper.

We also note that the first statement of Proposition
\ref{Z prime} has a more concrete reformulation:

\begin{Cor}    \label{form}
  Let $Z \in H$. For any $Z$-twisted $(G,q)$-oper whose
  $q$-connection $A(z)$ has the form \eqref{Ag} and any
  $Z'\in W_G\cdot Z$, $A(z)$ can be written in the
  form
\begin{equation}    \label{gaugeA1}
A(z)=n(qz)v(qz) \, Z' \, v(z)^{-1}n(z)^{-1}, \qquad n(z) \in N_-(z),
\quad v(z) \in B_+(z).
\end{equation}
\end{Cor}

\begin{proof} We give an independent proof. Let $Z'=w^{-1}Zw, w \in
  W_G$. Formula \eqref{Ag} implies
\begin{eqnarray}    \label{Ag1}
(g(qz)s)^{-1}A(z)(g(z)s)= Z',
\end{eqnarray}
where $s$ is a lifting of $w$ to $N(H)$. Since $Z' \in B_+ \subset
B_+(z)$, Theorem \ref{gen rel pos} implies that
$$
g(z)s = n(z)v(z), \qquad n(z) \in N_-(z), \quad
v(z) \in B_+(z).
$$
Substituting this back into formula \eqref{Ag1}, we obtain formula
\eqref{gaugeA1}.
\end{proof}

\subsection{The associated Cartan $q$-connection}    \label{Hqconn}

{}From now on, let $Z$ be a regular element of the maximal torus
$H \subset G$.

Consider a Miura $(G,q)$-oper with regular singularities determined by
polynomials $\{ \Lambda_i(z) \}_{i=1,\ldots,r}$. By Corollary
\ref{Miura form}, the underlying $(G,q)$-connection can be written in
the form \eqref{form of A}. Since it preserves the $B_+$-bundle
$\cF_{B_+}$ that is part of the data of this Miura $(G,q)$-oper (see
Definition \ref{Miura}), it may be viewed as a meromorphic
$(B_+,q)$-connection on $\P^1$. Taking the quotient of $\cF_{B_+}$ by
$N_+ = [B_+,B_+]$ and using the fact that $B/N_+ \simeq H$, we obtain
an $H$-bundle $\cF_{B_+}/N_+$ endowed with an
$(H,q)$-connection, which we denote by $A^H(z)$. According to formula
\eqref{form of A}, it is given by the formula
\begin{equation}    \label{AH}
A^H(z)=\prod_ig_i(z)^{\check{\alpha}_i}.
\end{equation}
We call $A^H(z)$ the \emph{associated Cartan $q$-connection} of the
Miura $q$-oper $A(z)$.

Now, if our Miura $q$-oper is $Z$-twisted (see Definition
\ref{ZtwMiura}), then we also have $A(z)=v(qz)Z v(z)^{-1}$, where
$v(z)\in B_+(z)$.  Since $v(z)$ can be written as
\begin{equation}    \label{vz}
v(z)=
\prod_i y_i(z)^{\check{\alpha}_i} n(z), \qquad n(z)\in N_+(z), \quad
y_i(z) \in \C(z)^\times,
\end{equation}
the Cartan $q$-connection $A^H(z)$ has the form
\begin{equation}    \label{AH1}
A^H(z)=\prod_i
y_i(qz)^{\check{\alpha}_i} \; Z \; \prod_i y_i(z)^{-\check{\alpha}_i}
\end{equation}
and hence we will refer to $A^H(z)$ as a $Z$-{\em twisted Cartan
  $q$-connection}. This formula shows that $A^H(z)$ is completely
determined by $Z$ and the rational functions $y_i(z)$. Indeed,
comparing this equation with \eqref{AH} gives
\begin{equation}    \label{giyi}
g_i(z)=\zeta_i\frac{y_i(qz)}{y_i(z)}\,.
\end{equation}

It is also the case that $A^H(z)$ determines the $y_i(z)$'s uniquely
up to scalar.  Indeed, if
$$
\prod_i
\wt{y}_i(qz)^{\check{\alpha}_i}  Z \prod_i
\wt{y}_i(z)^{-\check{\alpha}_i}=Z
$$
as well, then
$$
\prod_i
\left(\frac{\wt{y}_i(qz)}{y_i(qz)}\right)^{\check{\alpha}_i} Z  \prod_i
\left(\frac{\wt{y}_i(z)}{y_i(z)}\right) ^{-\check{\alpha}_i}=Z.
$$
The commutativity of $H(z)$ immediately implies that the rational
functions $h_i(z)=\frac{\wt{y}_i(z)}{y_i(z)}$ satisfy
$h_i(qz)=h_i(z)$.  Since $q$ is not a root of unity, we find that
$h_i(z)\in\C^\times$.

\section{Nondegenerate Miura-Pl\"ucker
  $q$-opers}    \label{Sec:nondegMiura}

Our main goal is to link Miura $q$-opers to solutions of a certain
system of equations called the $QQ$-system, which is in turn related
to the system of Bethe Ansatz equations. We do this in two steps:
first, we introduce the notion of $Z$-twisted Miura-Pl\"ucker
$(G,q)$-opers; these are Miura $q$-opers satisfying a slightly weaker
condition than the $Z$-twisted Miura $q$-opers discussed in the
previous section.  Second, we impose two \emph{nondegeneracy}
conditions on these $Z$-twisted Miura-Pl\"ucker $(G,q)$-opers.

We start in Section \ref{rank2} with an outline of the Pl\"ucker
description of $B_+$-bundles.  When $G$ has rank greater than $1$, we
use this to associate to a Miura $(G,q)$-oper a collection of Miura
$(\GL(2),q)$-opers indexed by the fundamental weights of $G$. This
will motivate the definition of $Z$-twisted Miura-Pl\"ucker
$(G,q)$-opers in Section \ref{MP}. We will then define two
nondegeneracy conditions for these objects: the first one in Section
\ref{H nondeg}, and the second one in Sections \ref{nondeg sl2} (for
$G=\SL(2)$) and \ref{s:nondegenerate} (for general $G$).

\subsection{The associated Miura $(\GL(2),q)$-opers}    \label{rank2}

In this section, we associate to a Miura $(G,q)$-oper with regular
singularities a collection of Miura $(\GL(2),q)$-opers indexed by the
fundamental weights. This is done using the Pl\"ucker description of
$B_+$-bundles which we learned from V. Drinfeld.

Recall that $\omega_i$ denotes the $i$th fundamental weight of $G$.
Let $V_i$ be the irreducible representation of $G$ with highest weight
$\omega_i$ with respect to $B_+$. It comes equipped with a line $L_i
\subset V_i$ (of highest weight vectors) stable under the action of
$B_+$. Likewise, there is a $B_+$-stable line $L_\la$ in the
irreducible representations $V_\la$ of $G$ for every dominant integral
highest weight $\la$. These lines satisfy the following generalized
{\em Pl\"ucker relations}: for any two dominant integral weights $\la$
and $\mu$, $L_{\la+\mu} \subset V_{\la+\mu}$ is the image of $L_\la
\otimes L_\mu \subset V_\la \otimes V_\mu$ under the canonical
projection $V_\la \otimes V_\mu \to V_{\la+\mu}$. Conversely, given a
collection of lines $L_\la \subset V_\la$ for all $\la$ satisfying the
Pl\"ucker relations, there is a Borel subgroup $B \subset G$ such that
$L_\la$ is stabilized by $B$ for all $\la$. A choice of $B$ is
equivalent to a choice of $B_+$-torsor in $G$. Hence, we can identify
the datum of a $B_+$-reduction $\cF_{B_+}$ of a $G$-bundle $\cF_G$
with the ``linear algebra'' data of line subbundles $\cL_i$ of the
associated vector bundles $\cV_\la = \cF \underset{G}\times V_\la$
satisfying the Pl\"ucker relations. If we have a connection or a
$q$-connection on $\cF_G$ that preserves $\cF_{B_+}$ (or has
another relation with $\cF_{B_+}$, such as the oper or $q$-oper
condition), we can use this formalism to express the properties of
$A(z)$ in terms of these ``linear algebra'' data.

In our discussion here, we will not make full use of this
formalism. What we need is the following simple fact. Let
$\nu_{\omega_i}$ be a generator of the line $L_i \subset V_i$. It is a
vector of weight $\omega_i$ with respect to our maximal torus $H
\subset B_+$. The subspace of $V_i$ of weight $\omega_i-\alpha_i$
is one-dimensional and is spanned by $f_i \cdot
\nu_{\omega_i}$. Therefore, the two-dimensional subspace $W_i$ of
$V_i$ spanned by the weight vectors $\nu_{\omega_i}$ and $f_i \cdot
\nu_{\omega_i}$ is a $B_+$-invariant subspace of $V_i$.

Now, let $(\cF_G,A,\cF_{B_-},\cF_{B_+})$ be a Miura $(G,q)$-oper
with regular singularities determined by polynomials $\{ \Lambda_i(z)
\}_{i=1,\ldots,r}$ (see Definition \ref{MiuraRS}). Recall that
$\cF_{B_+}$ is a $B_+$-reduction of a $G$-bundle $\cF_G$ on $\P^1$
preserved by the $(G,q)$-connection $A$. Therefore for each
$i=1,\ldots,r$, the vector bundle
$$
\cV_i = \cF_{B_+} \underset{B_+}\times V_i = \cF_G \underset{G}\times
V_i
$$
associated to $V_i$ contains a rank two
subbundle
$$
\cW_i = \cF_{B_+} \underset{B_+}\times W_i
$$
associated to $W_i \subset V_i$, and $\cW_i$ in
turn contains a line subbundle
$$
\cL_i = \cF_{B_+} \underset{B_+}\times L_i
$$
associated to $L_i \subset W_i$.

Denote by $\phi_i(A)$ the $q$-connection on the vector bundle $\cV_i$
(or equivalently, a $(\GL(V_i),q)$-connection) corresponding to the
above Miura $q$-oper connection $A$. Since $A$ preserves $\cF_{B_+}$
(see Definition \ref{Miura}), we see that $\phi_i(A)$ preserves the
subbundles $\cL_i$ and $\cW_i$ of $\cV_i$. Denote by $A_i$ the
corresponding $q$-connection on the rank 2 bundle $\cW_i$.

Let us trivialize $\cF_{B_+}$ on a Zariski open subset of $\P^1$ so
that $A(z)$ has the form \eqref{form of A} with respect to this
trivialization (see Corollary \ref{Miura form}). This trivializes the
bundles $\cV_i$, $\cW_i$, and $\cL_i$, so that the $q$-connection
$A_i(z)$ becomes a $2 \times 2$ matrix whose entries are in $\C(z)$.

Direct computation using formula \eqref{form of A} yields the
following result.

\begin{Lem}    \label{2flagthm}
We have
\begin{equation}    \label{2flagformula}
A_i(z)=\begin{pmatrix}
  g_i(z) &  &\Lambda_i(z) \prod_{j>i} g_j(z)^{-a_{ji}}\\
&&\\  
  0 & &g^{-1}_i(z) \prod_{j\neq i} g_j(z)^{-a_{ji}}
 \end{pmatrix},
\end{equation}
where we use the ordering of the simple roots determined by the
Coxeter element $c$.
\end{Lem}

Using the trivialization of $\cW_i$ in which $A_i(z)$ has the form
\eqref{2flagformula}, we represent $\cW_i$ as the direct sum of two
line subbundles. The first is $\cL_i$, generated by the basis vector
$\begin{pmatrix} 1 \\ 0 \end{pmatrix}$. The second, which we denote by
$\wt\cL_i$, is generated by the basis vector $\begin{pmatrix} 0 \\
  1 \end{pmatrix}$. The subbundle $\cL_i$ is $A_i$-invariant, whereas
the subbundle $\wt\cL_i$ and $A_i$ satisfy the following
$(\GL(2),q)$-oper condition.

\begin{Def}    \label{GL2}
  A \emph{$(\GL(2),q)$-oper} on $\P^1$ is a triple $(\cW,A,\wt\cL)$,
  where $\cW$ is a rank 2 bundle on $\P^1$, $A: \cW \to \cW^q$ is a
  meromorphic $q$-connection on $\cW$, and $\wt\cL$ is a line
  subbundle of $\cW$ such that the induced map $\bar{A}:\wt\cL \to
  (\cW/\wt\cL)^q$ is an isomorphism on a Zariski open dense subset of
  $\P^1$.

  A Miura \emph{$(\GL(2),q)$-oper} on $\P^1$ is a quadruple
  $(\cW,A,\wt\cL,\cL)$, where $(\cW,A,\wt\cL)$ is a $(\GL(2),q)$-oper
  and $\cL$ is an $A$-invariant line subbundle of $\cW$.
\end{Def}

Using this definition, one obtains an alternative definition of
(Miura) $(\SL(2),q)$-opers: these are the (Miura) $(\GL(2),q)$-opers
defined by the above triples (resp. quadruples) satisfying the
additional property that there exists an isomorphism between $\on{Det}
\cW$ and the trivial line bundle on a Zariski open dense subset of
$\P^1$ and under this isomorphism $\on{det}(A)$ is equal to the
identity.

Our quadruple $(\cW_i,A,\wt\cL_i,\cL_i)$ is clearly a Miura
$(\GL(2),q)$-oper. It is not clear whether it is an $(\SL(2),q)$-oper
because $\on{det} A_i(z)$ is not necessarily equal to 1 in the above
trivialization of $\cW_i$.

We now make a further assumption that our initial Miura $(G,q)$-oper
$(\cF_G,A,\cF_{B_-}$, $\cF_{B_+})$ with regular singularities is such
that the associated Cartan connection $A^H(z)$ has the form
\eqref{AH1}:
\begin{equation}    \label{AH2}
A^H(z)=\prod_i
y_i(qz)^{\check{\alpha}_i} \; Z \; \prod_i y_i(z)^{-\check{\alpha}_i},
\qquad y_i(z) \in \C(z).
\end{equation}
We claim that if this condition is satisfied, then there exists
another trivialization of $\cW_i$ in which the $q$-connection $A_i$
has a constant determinant (albeit not equal to 1, in general). In
other words, we can apply to $A_i(z)$ given by formula
\eqref{2flagformula} a $q$-gauge transformation so that the resulting
$q$-connection has constant determinant. This provides a particularly
convenient gauge for $A_i$.

To see this, let us apply to $A_i(z)$ the $q$-gauge transformation by
the diagonal matrix
$$
 u(z)= \begin{pmatrix} 1 & 0 \\ 0 & \prod_{j\neq i} y_j(z)^{a_{ji}}
\end{pmatrix}.
 $$
This gives
 \begin{equation}\label{cano}
\wt{A}_i(z) = u(qz)A_i(z)u^{-1}(z)=\begin{pmatrix}
   \zeta_i\frac{y_i(qz)}{y_i(z)} &  &\rho_i(z)\\
   0 & &\prod_{j\ne i} \zeta_j^{-a_{ji}} \zeta^{-1}_i\frac{y_i(z)}{y_i(qz)} 
  \end{pmatrix},
 \end{equation}
 where
 \begin{equation}\label{ri}
 \rho_i(z)=\Lambda_i(z)\prod_{j> i} (\zeta_j
 y_j(qz))^{-a_{ji}}\prod_{j<i}y_j(z)^{-a_{ji}}.
 \end{equation}
 Since $a_{ij} \leq 0$ for $i \neq j$, $\rho_i(z)$
 is a polynomial if all $y_j(z), j=1,\dots,r$, are polynomials.

Let $G_i\cong \SL(2)$ be the subgroup of $G$ corresponding to the
$\mathfrak{sl}(2)$-triple spanned by $\{e_i, f_i, \check{\alpha}_i\}$.
Note that the group $G_i$ preserves $W_i$. Consider the Miura
$(G_i,q)$-oper $(\cW_i,{\mc A}_i,$ $\wt\cL_i,\cL_i)$ with $\wt\cL_i =
\on{span} \left\{ \begin{pmatrix} 0 \\ 1 \end{pmatrix} \right\}$,
$\cL_i = \on{span} \left\{ \begin{pmatrix} 1 \\ 0 \end{pmatrix}
\right\}$, and
\begin{equation}    \label{Bi}
{\mc A}_i(z) = g_i^{\check \alpha_i}(z) \; e^{\frac{\rho_i(z)}{g_i(z)}e_i} =
\begin{pmatrix}
   \zeta_i\frac{y_i(qz)}{y_i(z)} &  &\rho_i(z)\\
   0 & &\zeta^{-1}_i\frac{y_i(z)}{y_i(qz)} 
  \end{pmatrix}.
\end{equation}
We can now express our $q$-connection $\wt{A}_i(z)$ given by formula
\eqref{cano} as follows:
\begin{align}    \label{tilde calig}
\wt{A}_i(z) &= \begin{pmatrix} 1 & 0 \\ 0 & \prod_{j\ne i}
  \zeta_j^{-a_{ji}}
\end{pmatrix} {\mc A}_i(z) \\    \label{Ai}
&= \begin{pmatrix} 1 & 0 \\ 0 & \prod_{j\ne i}
  \zeta_j^{-a_{ji}}
\end{pmatrix} g_i^{\check \alpha_i}(z) \;
e^{\frac{\rho_i(z)}{g_i(z)}e_i}.
\end{align}
This shows that $\on{det}(\wt{A}_i(z))$ is constant; namely, it is
equal to $\prod_{j\ne i} \zeta_j^{-a_{ji}}$.

Thus, under the assumption \eqref{AH2}, our Miura $(G,q)$-oper $A(z)$
gives rise to a collection of meromorphic Miura $(\SL(2), q)$-opers
${\mc A}_i(z)$ for $i=1,\ldots,r$. It should be noted that it has
regular singularities in the sense of Definition~\ref{d:regsing} if
and only if $\rho_i(z)$ is a polynomial. For example, this holds for
all $i$ if all $y_j(z), j=1,\dots,r$, are polynomials, an observation
we will use below.

\subsection{$Z$-twisted Miura-Pl\"ucker $q$-opers}    \label{MP}

Recall Definition \ref{ZtwMiura} of $Z$-twisted Miura $(G,q)$-opers,
where $Z$ is a regular semisimple element of the maximal torus
$H$. These are Miura $(G,q)$-opers whose underlying $q$-connection can
be written in the form \eqref{gaugeA}:
\begin{equation}    \label{gaugeA2}
A(z)=v(qz)Z v(z)^{-1}, \qquad v(z) \in B_+(z).
\end{equation}
We will now relax this condition using the Miura $(\GL(2),q)$-opers
$A_i(z)$ (or equivalently, the Miura $(\SL(2),q)$-opers ${\mc
  A}_i(z)$) associated to $A(z)$. Instead, we will require that there
exists an element $v(z)$ from $B_+(z)$ such that $A_i(z)$ satisfies
formula \eqref{gaugeA2} with $v(z)$ replaced by $v_i(z) = v(z)|_{W_i}
\in \GL(2)$ and $Z$ replaced by $Z|_{W_i}$ for all $i=1,\ldots,r$.

\begin{Def}    \label{ZtwMP}
  A $Z$-{\em twisted Miura-Pl\"ucker $(G,q)$-oper} is a meromorphic
  Miura $(G,q)$-oper on $\P^1$ with underlying $q$-connection $A(z)$
  satisfying the following condition: there exists $v(z) \in B_+(z)$
  such that for all $i=1,\ldots,r$, the Miura $(\GL(2),q)$-opers
  $A_i(z)$ associated to $A(z)$ by formula \eqref{2flagformula} can be
  written in the form
\begin{equation}    \label{gaugeA3}
A_i(z) = v(zq) Z v(z)^{-1}|_{W_i} = v_i(zq)Z_iv_i(z)^{-1},
\end{equation}
where $v_i(z) = v(z)|_{W_i}$ and $Z_i = Z|_{W_i}$.
\end{Def}

\medskip

The difference between $Z$-twisted Miura $(G,q)$-opers and $Z$-twisted
Miura-Pl\"ucker $(G,q)$-opers may be explained as follows: the former
is a quadruple $(\cF_G,A,$ $\cF_{B_-},\cF_{B_+})$ as in Definition
\ref{Miura} such that there exists a trivialization of $\cF_{B_+}$
with respect to which the $q$-connection $A$ is a constant element of
$G(z)$ equal to our element $Z \in H$. For the latter, we only ask
that there exists a trivialization of $\cF_{B_+}$ with respect to
which the $q$-connections $A_i(z)$ are constant elements of
$\GL(2)(z)$ equal to $Z_i$ for all $i=1,\ldots,r$.

Thus, every $Z$-twisted Miura $(G,q)$-oper is automatically a
$Z$-twisted Miura-Pl\"ucker $(G,q)$-oper, but the converse is not
necessarily true if $G \neq SL(2)$.

Note, however, that it follows from the above definition that the
$(H,q)$-connection $A^H(z)$ associated to a $Z$-twisted
Miura-Pl\"ucker $(G,q)$-oper can be written in the same form
\eqref{AH2} as the $(H,q)$-connection associated to a $Z$-twisted
Miura $(G,q)$-oper.

\subsection{$H$-nondegeneracy condition}    \label{H nondeg}

We now introduce two nondegeneracy conditions for $Z$-twisted
Miura-Pl\"ucker $q$-opers. The first of them, called the
$H$-nondegeneracy condition, is applicable to arbitrary Miura
$q$-opers with regular singularities. Recall from Corollary \ref{Miura
  form} that the underlying $q$-connection can be represented in the
form \eqref{form of A}.

In what follows, we will say that $v, w \in \mathbb{C}^\times$ are
\emph{$q$-distinct} if $q^\Z v\cap q^\Z w=\varnothing$.

\begin{Def}    \label{nondeg Cartan}
  A Miura $(G,q)$-oper $A(z)$ of the form \eqref{form of A} is called
  $H$-\emph{nondegene\-rate} if the corresponding $(H,q)$-connection
  $A^H(z)$ can be written in the form \eqref{AH1}, where for all
  $i,j,k$ with $i\ne j$ and $a_{ik} \neq 0, a_{jk} \neq
    0$, the zeros and poles of $y_i(z)$ and $y_j(z)$ are
  $q$-distinct from each other and from the zeros of
  $\Lambda_k(z)$.
\end{Def}

\subsection{Nondegenerate $Z$-twisted Miura
  $(\SL(2),q)$-opers}    \label{nondeg sl2}

Next, we define the second nondegeneracy condition. This condition
applies to $Z$-twisted Miura-Pl\"ucker $(G,q)$-opers. In this
subsection, we give the definition for $G=\SL(2)$. (Note that
$Z$-twisted Miura-Pl\"ucker $(\SL(2),q)$-opers are the same as
$Z$-twisted Miura $(\SL(2),q)$-opers.) In the next subsection, we will
give it in the case of an arbitrary simply connected simple complex
Lie group $G$.

Consider a Miura $(\SL(2),q)$-oper given by formula \eqref{form of A},
which reads in this case:
$$
A(z)=g(z)^{\check{\alpha}} \; \on{exp}\left( \frac{\Lambda(z)}{g(z)}e
\right) = \begin{pmatrix}
   g(z) & \Lambda(z) \\
   0 & g(z)^{-1}
  \end{pmatrix}.
$$
According to formula \eqref{AH2}, the corresponding Cartan
$q$-connection $A^H(z)$ is equal to
$$
A^H(z)=g(z)^{\check{\alpha}} = y(qz)^{\check{\alpha}}Z
y(z)^{-\check{\alpha}} = \begin{pmatrix}
   y(qz)y(z)^{-1} & 0 \\
   0 & y(qz)^{-1}y(z)
  \end{pmatrix},
$$
where $y(z)$ is a rational function. Let us assume that $A(z)$ is
$H$-nondegenerate (see Definition \ref{nondeg Cartan}). This means
that the zeros of $\Lambda(z)$ are $q$-distinct from the zeros and
poles of $y(z)$.

If we apply a $q$-gauge transformation by an element of
$h(z)^{\check\alpha} \in H[z]$ to $A(z)$, we obtain a new $q$-oper
connection
\begin{equation}    \label{wtA}
\wt{A}(z) = \wt{g}(z)^{\check{\alpha}} \; \on{exp}\left(
  \frac{\wt\Lambda(z)}{\wt{g}(z)}e \right),
\end{equation}
where
\begin{equation}    \label{wtg}
\wt{g}(z) = g(z) h(zq) h(z)^{-1}, \qquad \wt\Lambda(z) =
\Lambda(z) h(zq) h(z).
\end{equation}
It also has regular singularities, but for a
different polynomial $\wt{\Lambda}(z)$, and $\wt{A}(z)$ may no longer
be $H$-nondegenerate.  However, it turns out there is an essentially
unique gauge transformation from $H[z]$ for which the resulting
$\wt{A}(z)$ is $H$-nondegenerate  and $\wt{y}(z)$
is a polynomial.  This choice allows us to fix the polynomial
$\Lambda(z)$ determining the regular singularities of our
$(\SL(2),q)$-oper.

\begin{Lem}    \label{nondegsl2}
\begin{enumerate}
\item There is an $H$-nondegenerate $(\SL(2),q)$-oper $\wt{A}(z)$ in
  the $H[z]$-gauge class of $A(z)$, say with
  $\wt{A}^H(z)=\wt{g}(z)^{\check{\alpha}}$, for which the rational
  function $\wt{y}(z)$ is a polynomial. This oper is unique up to a
  scalar $a \in \C^\times$ that leaves $\wt{g}(z)$ unchanged, but
  multiplies $\wt{y}(z)$ and $\wt{\Lambda}(z)$ by $a$ and $a^2$
  respectively.
\item This $(\SL(2),q)$-oper $\wt{A}(z)$ may also be characterized by
  the property that $\wt{\Lambda}(z)$ has maximal degree subject to
  the constraint that it is $H$-nondegenerate.
\end{enumerate}
\end{Lem}

\begin{proof} Write $y(z)=\frac{P_1(z)}{P_2(z)}$, where $P_1,P_2$ are
  relatively prime polynomials. For a nonzero polynomial $h(z) \in
  \C(z)^\times$, the gauge transformation of $A(z)$ by
  $h(z)^{\check{\alpha}}$ is given by formulas \eqref{wtA} and
  \eqref{wtg}. Thus, in order for
  $\wt{y}(z)=h(z)\frac{P_1(z)}{P_2(z)}$ to be a polynomial, we need
  $h(z)$ to be divisible by $P_2(z)$. If, however, $\deg(h/P_2)>0$,
  then $\wt{y}(z)$ and $\wt{\Lambda}(z)$ would have a zero in common,
  so $\wt{A}(z)$ would not by $H$-nondegenerate.  Hence, we must have
  $h(z)=aP_2(z)$ for some $a\in\C^\times$. Thus, $h(z)$ is uniquely
  defined by multiplication by $a$, which leave $\wt{g}(z)$ unchanged,
  but multiplies $\wt{y}(z)$ and $\wt{\Lambda}(z)$ by $a$ and $a^2$
  respectively.

  For the second statement, note that if $h(z)$ is a polynomial for
  which the zeros of $h(z)h(qz)\Lambda(z)$ are $q$-distinct from the
  zeros and poles of $h(z)\frac{P_1(z)}{P_2(z)}$, we must have
  $h|P_2$.  If $h(z)$ is not an associate of $P_2(z)$, we have
  $\deg(h)<\deg(P_2)$, so
  $\deg(h(z)h(qz)\Lambda(z))<\deg(\wt{\Lambda})$.
\end{proof}

This motivates the following definition.

\begin{Def}    \label{ngsl2}
  A $Z$-twisted Miura $(\SL(2),q)$-oper is called \emph{nondegenerate}
  if it is $H$-nondegenerate and the rational function $y(z)$
  appearing in formula \eqref{AH1} is a polynomial.
\end{Def}

\subsection{Nondegenerate  Miura-Pl\"ucker
  $(G,q)$-opers}    \label{s:nondegenerate}

We now turn to the general case. Recall Definition \ref{ZtwMP} of
$Z$-twisted Miura-Pl\"ucker $(G,q)$-opers. Further recall that to every
Miura $(G,q)$-oper $A(z)$, we have associated a Miura $(\SL(2),q)$-oper
${\mc A}_i(z), i=1,\ldots,r$, given by formula \eqref{Bi}. (It can be
obtained from the Miura $(\GL(2),q)$-oper $A_i(z) = A(z)|_{W_i}$ using
formulas \eqref{cano} and \eqref{tilde calig}). It follows from the
definition that if $A(z)$ is $Z$-twisted  with $Z$ given by \eqref{Z},
then ${\mc A}_i(z)$ is $\check\alpha_i(\zeta_i)$-twisted.

\begin{Def}    \label{nondeg Miura}
  Suppose that the rank of $G$ is greater than 1. A $Z$-twisted
  Miura-Pl\"ucker $(G,q)$-oper $A(z)$ is called \emph{nondegenerate} if
  it is $H$-nondegenerate and
  each $\check\alpha_i(\zeta_i)$-twisted Miura
  $(\SL(2),q)$-oper ${\mc A}_i(z)$ is nondegenerate.
\end{Def}

It turns out that this simply means that in addition to $A(z)$ being
$H$-nondegenerate, each $y_i(z)$ from formula \eqref{AH1} is a polynomial
satisfying a mild condition on its roots.

\begin{Prop}    \label{nondeg1}
  Suppose that the rank of $G$ is greater than 1, and let $A(z)$ be a
  $Z$-twisted Miura-Pl\"ucker $(G,q)$-oper.
  The following statements are equivalent:
\begin{enumerate}
\item\label{nondegen1} $A(z)$ is nondegenerate.
  \item\label{nondegen2} $A(z)$ is $H$-nondegenerate, and each
    ${\mc A}_i(z)$ has regular singularities, i.e. $\rho_i(z)$ given
    by formula \eqref{ri} is in $\C[z]$.
    \item\label{nondegen3} Each $y_i(z)$ from formula \eqref{AH1} may
      be chosen to be a monic
      polynomial, and for all $i,j,k$ with $i\ne
      j$ and $a_{ik} \neq 0, a_{jk} \neq
    0$, the zeros of $y_i(z)$ and $y_j(z)$ are
  $q$-distinct from each other and from the zeros of
  $\Lambda_k(z)$.
\end{enumerate}
\end{Prop}

\begin{proof} To prove that \eqref{nondegen2} implies
  \eqref{nondegen3}, we need only show that if each $\rho_i(z)$ given
  by formula \eqref{ri} is in $\C[z]$, then the $y_i(z)$'s are
  polynomials.  Suppose $y_i(z)$ is not a polynomial, and choose $j\ne
  i$ such that $a_{ij} \neq 0$.  Then $-a_{ij}>0$ and so the
  denominator of $y_i(z)$ or $y_i(qz)$ appears in the denominator of
  $\rho_j(z)$.  Moreover, since the poles of $y_i(z)$ are $q$-distinct
  from the zeros of $\Lambda_j(z)$ and the other $y_k(z)$'s, the poles
  of $y_i(z)$ or $y_i(qz)$ would give rise to poles of
  $\rho_j(z)$. But then ${\mc A}_j(z)$ would not have regular
  singularities.

  Next, assume \eqref{nondegen3}. Then $A(z)$ is $H$-nondegenerate by
  Definition \ref{nondeg Cartan}.  Since all the $y_i(z)$'s are
  polynomials, the same is true for the $\rho_i(z)$'s.  (Here, we
  are using the fact that the off-diagonal elements of the Cartan
  matrix, $a_{ij}$ with $i\neq j$, are less than or equal to 0.)
  Since $\rho_i(z)$ is a product of polynomials whose roots are
  $q$-distinct from the roots of $y_i(z)$, we see that the Cartan
  $q$-connection associated to ${\mc A}_i(z)$ is nondegenerate.

Finally, \eqref{nondegen2} is a trivial consequence of
\eqref{nondegen1}.
\end{proof}

If we apply a $q$-gauge transformation by an element $h(z)\in H[z]$ to
$A(z)$, we get a new $Z$-twisted Miura-Pl\"ucker $(G,q)$-oper.
However, the following proposition shows that it is only nondegenerate
if $h(z)\in H$.  As a consequence, the $\Lambda_k$'s of a
nondegenerate $q$-oper are determined up to scalar multiples. If we further impose the condition that each $y_i(z)$ is a {\em monic}
polynomial, then $h(z)=1$, and this fixes the $\Lambda_k$'s.

\begin{Prop} If $A(z)$ is a nondegenerate $Z$-twisted Miura-Pl\"ucker
  $(G,q)$-oper and $h(z)\in H[z]$, then $h(qz)A(z)h(z)^{-1}$ is
  nondegenerate if and only if $h(z)$ is a constant element of $H$.
\end{Prop}

\begin{proof} Write $h(z)=\prod h_i(z)^{\check{\alpha}_i}$.  Gauge
  transformation of $A(z)$ by $h(z)$ induces a gauge transformation of
  $A_i(z)$ by $h_i(z)$.  Since $A_i(z)$ is nondegenerate,
  Lemma~\ref{nondegsl2} implies that the new Miura $(\SL(2),q)$-oper
  is nondegenerate if and only $h_i\in\C^\times$.
\end{proof}

\subsection{Dependence on the Coxeter element}

We end this section with a preliminary result on the dependence of our
results on the specific Coxeter element fixed in the definition of
$q$-opers.  We will see later in Section~\ref{equiv} that the
$QQ$-systems obtained from different choices of Coxeter element are
equivalent.  Here, we show that if two Coxeter elements $c$ and $c'$
are related by a cyclic permutation of their simple reflection
factors, then the corresponding spaces of $(G,q)$-opers with regular
singularities are isomorphic via a map defined in terms of
$B_+(z)$-gauge transformations.  Moreover, this map preserves
nondegeneracy.

\begin{Prop}\label{prop:coxchoice}
  Let $c$ and $c'$ be two Coxeter elements that differ by a cyclic
  permutation of their simple reflection factors.  Then, there is an
  isomorphism between the spaces of $Z$-twisted Miura $(G,q)$-opers
  with regular singularities defined in terms of $c$ and $c'$ of the
  form $A(z)\mapsto f_A(qz) A(z) f_A(z)^{-1}$, where $f_A\in B_+(z)$.
  This isomorphism takes nondegenerate opers to nondegenerate opers.
\end{Prop}

\begin{proof} Without  loss of generality, we may assume that
  $c=w_{i_1}\dots w_{i_r}$ and $c'=w_{i_2}\dots w_{i_r}w_{i_1}$.
  Given
$$
A(z)=\prod_{j=1}^r
  g_{i_j}(z)^{\check{\alpha}_{i_j}} \;
  e^{\frac{\Lambda_{i_j}(z)}{g_{i_j}(z)}e_{i_j}},
$$
set
$$
f_A(qz)=\left(g_{i_1}(z)^{\check{\alpha}_{i_1}}
e^{\frac{\Lambda_{i_1}(z)}{g_{i_1}(z)}e_{i_1}}\right)^{-1}.
$$
The effect of gauge transformation by $f_A(z)$ is to move the
$q^{-1}$-shift of the $i_1$ component of $A$ to the end of the
product, thereby giving the order corresponding to $c'$.  The new
$y_i$'s and $\Lambda_i$'s are the same except for the $q^{-1}$-shift
of $y_{i_1}$ and $\Lambda_{i_1}$, so it is obvious that the new
$q$-oper also has regular singularities and is nondegenerate if the
original $q$-oper was. It is also clear that this map is an
isomorphism.
\end{proof}

\section{$(SL(2), q)$-opers and the Bethe Ansatz equations}
\label{Sec:SL2review}

Our goal is to establish a bijection between the set of nondegenerate
$Z$-twisted Miura-Pl\"ucker $(G,q)$-opers and the set of nondegenerate
solutions of a system of Bethe Ansatz equations. In this section, we
show this for $G=\SL(2)$, which corresponds to the XXZ
model. This was already shown in \cite{KSZ}, in which a slightly
different definition of $(\SL(2),q)$-opers was used. Below, we explain
the connection to the formalism used in \cite{KSZ}.

\subsection{From non-degenerate $(\SL(2), q)$-opers to the
  $QQ$-system}

Suppose we have a $Z$-twisted nondegenerate Miura (equivalently, a
Miura-Pl\"ucker) $(\SL(2),q)$-oper. As explained in Section
\ref{nondeg sl2}, the underlying $q$-connection may be written in the
form
$$
A(z) = \begin{pmatrix}
   g(z) & \Lambda(z) \\
   0 & g(z)^{-1}
  \end{pmatrix},
$$
and furthermore, there exists $v(z) \in B_+(z)$ such that
\begin{equation}    \label{Azvz}
A(z) = v(zq) Z v(z)^{-1}, \qquad Z = \begin{pmatrix}
   \zeta & 0 \\
   0 & \zeta^{-1}
  \end{pmatrix}.
\end{equation}
Write
\begin{equation}    \label{exp vz}
v(z) =  \begin{pmatrix}
   y(z) & 0 \\
   0 & y(z)^{-1}
  \end{pmatrix} \begin{pmatrix}
   1 & - \frac{Q_-(z)}{Q_+(z)} \\
   0 & 1
  \end{pmatrix} = \begin{pmatrix}
   y(z) & - y(z) \frac{Q_-(z)}{Q_+(z)} \\
   0 & y(z)^{-1}
  \end{pmatrix},
\end{equation}
where $Q_+(z)$ and $Q_-(z)$ are relatively prime polynomials such that
$Q_+(z)$ is a monic polynomial. Formula \eqref{Azvz} then yields
$$
g(z) = \zeta_i y(zq)y(z)^{-1}
$$
and
\begin{equation}    \label{Lay}
\Lambda(z) = y(z) y(zq) \left( \zeta \frac{Q_-(z)}{Q_+(z)} - \zeta^{-1}
\frac{Q_-(zq)}{Q_+(zq)} \right).
\end{equation}
Nondegeneracy (see Definition \ref{ngsl2}) means that
$\Lambda(z)$ and $y(z)$ are polynomials whose roots are $q$-distinct
from each other. This can only be satisfied if $y(z)$ equals a scalar
multiple of $Q_+(z)$. Since we have the freedom to rescale $y(z)$,
without loss of
generality we can and will assume that $y(z)=Q_+(z)$. Equation
\eqref{Lay} then becomes
\begin{equation}    \label{QQ sl2}
\zeta Q_-(z) Q_+(zq) - \zeta^{-1} Q_-(zq) Q_+(z) = \Lambda(z).
\end{equation}
We call equation \eqref{QQ sl2} the $QQ$-{\em system} associated to
$\SL(2)$. (See the last paragraph of Section \ref{QQBethesl2} and
Section \ref{prior} for a discussion of the origins of this system in
the XXZ model.) Here, $\Lambda(z)$ is fixed: it is the polynomial used
in the definition of a Miura $(\SL(2),q)$-opers which contains the
information about their regular singularities.  Thus, the $QQ$-system
is an equation on two polynomials $Q_+(z), Q_-(z)$.

Let us call a solution $\{ Q_+(z), Q_-(z) \}$ of \eqref{QQ sl2} {\em
  nondegenerate} if $Q_+(z)$ is a monic polynomial whose roots are
$q$-distinct from the roots of the polynomial $\Lambda(z)$. No
conditions are imposed on $Q_-(z)$, but note that the nondegeneracy
condition and formula \eqref{QQ sl2} imply that $Q_+(z)$ and $Q_-(z)$
are {\em relatively prime}. The above discussion is summarized in the
following statement.

\begin{Thm}    \label{isom sl2}
  There is a one-to-one correspondence between the set of
  nondegenerate $Z$-twisted $(\SL(2),q)$-opers with regular
  singularity determined by a polynomial $\Lambda(z)$ and the
  set of nondegenerate solutions of the $QQ$-system \eqref{QQ sl2}.
\end{Thm}

\subsection{From the $QQ$-system to the Bethe Ansatz equations}
\label{QQBethesl2}

Under our assumption that $Q_+(z)$ is a monic polynomial, we can write
$$
Q_+(z) = \prod_{k=1}^m (z-w_k).
$$
Evaluating \eqref{QQ sl2} at $q^{-1}z$, we get
$$
\Lambda(q^{-1}z)=\zeta Q_{-}(q^{-1}z)Q_{+}(z)-\zeta^{-1}
Q_{-}(z)Q_{+}(q^{-1}z).
$$
If we divide \eqref{QQ sl2} by this equation and evaluate at the
roots $w_k$ of $Q_+(z)$, we obtain the following equations:
\begin{equation}    \label{BAE sl2}
\frac{\Lambda(w_k)}{\Lambda(q^{-1}w_k)}=
-\zeta^{2}\frac{Q_+(qw_k)}{Q_+(q^{-1}w_k)}, \qquad k=1,\ldots,m.
\end{equation}

These equations are equivalent to the Bethe Ansatz equations of the
XXZ model, i.e., the quantum spin chain associated to $U_q
\widehat\sl_2$. To express them in a more familiar form, suppose that
$\Lambda(z)$ is a monic polynomial all of whose roots are non-zero and
simple.  Recalling that we do not require the roots of $\Lambda(z)$ to
be mutually $q$-distinct, we write $\Lambda(z)$ explicitly as
\begin{equation}    \label{Lambda sl2}
\Lambda(z) = \prod^{L}_{p=1}\prod^{r_p-1}_{j_p=0}(z-q^{-j_p}z_p),
\end{equation}
where the $z_p$'s are mutually $q$-distinct and non-zero. Setting $r =
\sum_{p=1}^L t_p$, the equations \eqref{BAE sl2} become
\begin{equation}    \label{sl2qbethe}
q^r \prod_{p=1}^L\frac{w_k-q^{1-r_p}z_p}{w_k-q
    z_p}=-\zeta^{2}q^m \prod^m_{j=1}\frac{q
    w_k-w_j}{w_k-q w_j},\qquad k=1,\dots,m.
\end{equation}
This is a more familiar form of the Bethe Ansatz equations in the XXZ
model (see e.g. \cite{Frenkel:2013uda}, Section 5.6).

Let us call a solution $Q_+(z)$ of the system of Bethe Ansatz
equations \eqref{BAE sl2} {\em nondegenerate} if $Q_+(z)$ is a monic
polynomial whose roots are $q$-distinct from the roots of
$\Lambda(z)$. It is clear that if $\{ Q_+(z),Q_-(z) \}$ is a
nondegenerate solution of \eqref{QQ sl2}, then $Q_+(z)$ is a
nondegenerate solution of \eqref{BAE sl2}, and vice versa. The above
calculation, combined with Theorem \ref{isom sl2}, proves the
following result.

\begin{Thm}    \label{BASL2}
  There is a one-to-one correspondence between the set of
  nondegenerate $Z$-twisted $(\SL(2),q)$-opers with regular
  singularity determined by a polynomial $\Lambda(z)$ and the set of
  nondegenerate solutions of the Bethe Ansatz equations \eqref{BAE
    sl2}.
\end{Thm}

It is known that the Bethe Ansatz equations \eqref{BAE sl2} parametrize
the spectra of the quantum transfer-matrices in the XXZ model
corresponding to $U_{q'} \wh\sl_2$, where $q'=q^{-2}$, with the space
of states being the tensor product of finite-dimensional
representations of $U_{q'} \wh\sl_2$ (see
e.g. \cite{Frenkel:2013uda}). The polynomial $\Lambda(z)$ is the
product of the Drinfeld polynomials of these representations, up to
multiplicative shifts by powers of $q$. Furthermore, we expect
that the $QQ$-system \eqref{QQ sl2} can be
derived from the $Q\wt{Q}$-relation in the Grothendieck ring of the
category ${\mc O}$ of $U_{q'} \wh\sl_2$ proved in \cite{Frenkel:ac}.

\subsection{An approach using the $q$-Wronskian}

In \cite{KSZ}, the equations \eqref{QQ sl2} and \eqref{BAE sl2} were
derived in a slightly different way, and analogous results were also
obtained for $G=\SL(n)$. We now make an explicit connection between
this approach and the approach of the preceding section.

Recall Definition \ref{GL2} of $(\GL(2),q)$-opers. Adding the
condition that the underlying rank two vector bundle $\cW$ can be
identified with the trivial line bundle so that $\on{det}(A)=1$, we
obtain the definition of Miura $(\SL(2),q)$-opers. The oper condition
is now expressed as the existence of a line subbundle $\wt\cL \subset
\cW$ for which $\bar{A}: \wt\cL \to \cW/\wt\cL$ is an isomorphism on a
open dense subset of $\P^1$. Choose any trivialization of $\cW$
on an open dense subset $U$, and let $s(z)$ be a section of $\cW$ on
this subset that generates the line subbundle $\wt\cL$. The
$q$-connection $A(z)$ then satisfies the condition
$$
s(qz)\wedge A(z) s(z)\neq 0
$$
on a Zariski open dense subset $V$ of $U$. This is the definition of a
general meromorphic $(\SL(2),q)$-oper.

From this perspective, $(\SL(2),q)$-opers with regular singularities are
defined in \cite{KSZ} as follows.

\begin{Def}
  An \emph{$(\SL(2),q)$-oper with regular singularities determined by
    $\Lambda(z)$} is a meromorphic $(\SL(2),q)$-oper $(\cE,A,\wt\cL)$
  such that $s(qz)\wedge A(z) s(z) = \Lambda(z)$.
\end{Def}
This definition is equivalent to Definition \ref{d:regsing}.

Consider a diagonal matrix $Z=\diag(\ze,\ze^{-1})$ with $\ze\ne\pm 1$.
Recall that an $(\SL(2),q)$-oper $(\cE,A,\wt\cL)$ is a
\emph{$Z$-twisted $q$-oper} if $A$ is gauge equivalent to $Z$.  (We
remark that in \cite{KSZ}, a $Z$-twisted $q$-oper was actually defined
to be one that is gauge equivalent to $Z^{-1}$.  This does not matter
at the level of $q$-opers, but does matter if we consider the
corresponding Miura $q$-opers.)

Now, suppose that $(\cE, A,\mathcal{L})$ is a $Z$-twisted
$(\SL(2),q)$-oper with regular singularities determined by
a monic polynomial $\Lambda(z)$. (Note that we can assume that
$\Lambda(z)$ is monic after multiplying the section $s$ by a nonzero
constant.) Choose a trivialization of $\cE$ with respect to
which the $q$-connection matrix is $Z$. Since $\wt\cL$ can be
trivialized on $\P^1\setminus\infty$, it is generated by a section
\begin{equation}    \label{sz}
  s(z)=\begin{pmatrix} Q_-(z)\\Q_+(z)
  \end{pmatrix},
\end{equation}
where $Q_+(z)$ and $Q_-(z)$ are relatively prime polynomials and
$Q_+(z)$ is monic. Furthermore, the polynomials $Q_+(z), Q_-(z)$
satisfying these conditions are uniquely determined by $\wt\cL$.

Regular singularity of the $q$-oper then becomes an explicit equation
for the \emph{$q$-Wronskian} of $Q_-(z)$ and $Q_+(z)$:
\begin{equation}    \label{eq:qOpDefSL2}
  \zeta Q_{-}(z)Q_{+}(qz)-\zeta^{-1}
  Q_{-}(qz)Q_{+}(z) = \Lambda(z).
\end{equation}
This is just the $QQ$-system \eqref{QQ sl2}.

Note that in \cite{KSZ} it was assumed that $\Lambda(z)$ has the form
\eqref{Lambda sl2}, i.e. that the polynomial $\Lambda(z)$ only has
simple roots, and 0 is not a root. However, the same derivation of
\eqref{eq:qOpDefSL2} works for any monic polynomial $\Lambda(z)$.

Next, we explain the link between the section $s(z)$ and the $q$-oper
\eqref{Azvz}. Recall that a Miura $(\SL(2),q)$-oper in the sense of
Definition \ref{GL2} is a quadruple $(\cW,A,\wt\cL,\cL)$, where
$\wt\cL$ is a line subbundle satisfying the $q$-oper condition with
respect to $A$ and $\cL$ is an $A$-invariant line subbundle.

In the particular trivialization of $\cW$ that we are now considering,
the $q$-connection $A$ is equal to $Z$, $\wt\cL$ is generated by the
section \eqref{sz}, and $\cL$ is generated by
$\left( \begin{smallmatrix} 1 \\
    0 \end{smallmatrix} \right)$. To bring it to the form
\eqref{Azvz}, we need to change the trivialization of $\cW$ in such a
way that $\cL$ is still generated by $\left( \begin{smallmatrix} 1 \\
    0 \end{smallmatrix} \right)$, and $\wt\cL$ is generated
by$\left( \begin{smallmatrix} 0 \\
    1 \end{smallmatrix} \right)$. With respect to this new
trivialization, the oper $q$-connection becomes equal to the $q$-gauge
transformation of $Z$ by the corresponding element $U(z) \in G(z)$.
The above conditions means that $U(z)$ should preserve $\cL$, i.e. it
should be in $B_+(z)$ and should satisfy
$$
U(z)s(z)=\left(\begin{matrix}0\\1
  \end{matrix}\right).
$$

There is a unique such $U(z)$, namely, 
\begin{equation}
U(z)=\begin{pmatrix}
   Q_+(z) & -Q_-(z)\\
   0 &  Q_+^{-1}(z)
 \end{pmatrix}.
\end{equation}
Applying the $q$-gauge transformation by $U(z)$ to $Z$, we obtain a
formula for the oper $q$-connection $A(z)$ in the new trivialization
of $\cW$:
\begin{align}    \label{acon}
A(z)&=\begin{pmatrix} Q_+(qz) & - Q_-(qz)\\
   0 & Q_+^{-1}(qz)
 \end{pmatrix} \begin{pmatrix} \zeta & 0 \\
   0 & \zeta^{-1}
 \end{pmatrix}
\begin{pmatrix}  Q_+(z)^{-1} & Q_-(z)\\
   0 &  Q_+(z)
 \end{pmatrix}\\    \label{acon1}
&=\begin{pmatrix} \zeta Q_+(qz)Q^{-1}_+(z)&\Lambda(z)\\
   0 &  \zeta^{-1} Q_+^{-1}(qz)Q_+(z)
 \end{pmatrix},
\end{align}
where $\Lambda(z)$ is the $q$-Wronskian \eqref{eq:qOpDefSL2}.

Thus, we have arrived at a nondegenerate $Z$-twisted Miura
$(\SL(2),q)$-oper in the sense of Section \ref{Sec:nondegMiura}:
$A(z)=g^{\check\alpha}(z)e^{\frac{\Lambda(z)}{g(z)}e}$, where $g(z) =$
$\zeta Q_+(zq)Q_+(z)^{-1}$.

\section{Miura-Pl\"ucker $q$-opers, $QQ$-system, and Bethe Ansatz
  equations}    \label{Sec:QQsystem}

In this section, we generalize the results of the previous section to
an arbitrary simply connected simple complex Lie group $G$. We
establish a one-to-one correspondence between the set of nondegenerate
$Z$-twisted Miura-Pl\"ucker $(G,q)$-opers and the set of nondegenerate
solutions of a system of Bethe Ansatz equations associated to $G$. A
key element of the construction is an intermediate object between
these two sets: the set of nondegenerate solutions of the so-called
$QQ$-system.

\subsection{Miura $(G,q)$-opers and the $QQ$-system}

First, we construct a one-to-one correspondence between the set of
nondegenerate $Z$-twisted Miura-Pl\"ucker $(G,q)$-opers and the set of
nondegenerate solutions of the $QQ$-system.

Recall that we have chosen a set of non-zero polynomials $\{
\Lambda_i(z) \}_{i=1,\ldots,r}$, which we will assume from now on to
be monic, and a set of non-zero complex numbers $\{ \zeta_i
\}_{i=1,\ldots,r}$ that correspond to a regular element $Z$ of the
maximal torus $H \subset G$ by formula \eqref{Z}. In this section,
these data are assumed to be fixed. (In the next section,
we will also consider elements $w(Z)$ of the orbit of $Z$ under the
action of the Weyl group $W_G$ of $G$ and the corresponding
$\zeta_i$'s.)

From now on, we will assume that our element $Z = \prod_i
\zeta_i^{\check\alpha_i} \in H$ satisfies the following property:
\begin{equation}    \label{assume}
\prod_{i=1}^r \zeta_i^{a_{ij}} \notin q^\Z, \qquad
\forall j=1,\ldots,r\,.  
\end{equation}
Since $\prod_{i=1}^r \zeta_i^{a_{ij}}\ne 1$ is a special case of
\eqref{assume}, this implies that $Z$ is {\em regular semisimple}.

Introduce the following system of equations:
\begin{multline}\label{qq}
\wt{\xi}_iQ^i_{-}(z)Q^i_{+}(qz)-\xi_iQ^i_{-}(qz)Q^i_{+}(z) = \\
\Lambda_i(z)\prod_{j> i}\Big[Q^j_{+}(qz)\Big]^{-a_{ji}}
\prod_{j< i}\Big[Q^j_{+}(z)\Big]^{-a_{ji}}, \qquad
i=1,\ldots,r,
\end{multline}
where
\begin{equation}    \label{xi}
\wt{\xi}_i=\zeta_i \prod_{j>i} \zeta_j^{a_{ji}}, \qquad
{\xi}_i=\zeta^{-1}_i\prod_{j< i} \zeta_j^{-a_{ji}}
\end{equation}
and we use the ordering of simple roots from the definition of
$(G,q)$-opers.

We call this the $QQ$-{\em system} associated to $G$ and a collection
of polynomials $\Lambda_i(z)$, $i=1,\ldots,r$.

A polynomial solution $\{ Q^i_+(z),Q^i_-(z) \}_{i=1,\ldots,r}$ of
\eqref{qq} is called {\em nondegenerate} if it has the following
properties: condition \eqref{assume} holds for the $\zeta_i$'s; for
all $i,j,k$ with $i \neq j$ and $a_{ik}, a_{jk} \neq 0$, the
zeros of $Q^j_+(z)$ and $Q^j_-(z)$ are $q$-distinct from each other
and from the zeros of $\Lambda_k(z)$; and the polynomials $Q^i_+(z)$
are monic.

Recall Definition \ref{ZtwMP} of nondegenerate $Z$-twisted
Miura-Pl\"ucker $(G,q)$-opers.

\begin{Thm}    \label{inj}
  There is a one-to-one correspondence between the set of
  nondegenerate $Z$-twisted Miura-Pl\"ucker $(G,q)$-opers and the set
  of nondegenerate polynomial solutions of the $QQ$-system \eqref{qq}.
\end{Thm}

\begin{proof}
  Let $A(z)$ be a nondegenerate $Z$-twisted Miura-Pl\"ucker
  $(G,q)$-oper. According to Corollary \ref{Miura form}, it can be
  written in the form \eqref{form of A}:
\begin{equation}    \label{form of A1}
A(z)=\prod_i
g_i(z)^{\check{\alpha}_i} \; e^{\frac{\Lambda_i(z)}{g_i(z)}e_i}, \qquad
g_i(z) \in \C(z)^\times,
\end{equation}
and there exists $v(z) \in B_+(z)$ such that for all $i=1,\ldots,r$,
the Miura $(\GL(2),q)$-opers $A_i(z)$ associated to $A(z)$ by formula
\eqref{2flagformula} can be written in the form \eqref{gaugeA3}:
\begin{equation}    \label{gaugeA4}
A_i(z) = v_i(zq)Z_iv_i(z)^{-1}, \qquad i=1,\ldots,r,
\end{equation}
where $v_i(z) = v(z)|_{W_i}$ and $Z_i = Z|_{W_i}$.

The element $v(z)$ can be expressed in the form 
\begin{equation}    \label{vdots}
v(z) = \prod_{i=1}^r y_i(z)^{\check\alpha_i} \prod_{i=1}^r
e^{-\frac{Q^i_{-}(z)}{Q^i_{+}(z)} e_i} \dots ,
\end{equation}
where the dots stand for the exponentials of higher commutator terms
in ${\mathfrak n}_+=\Lie N_+$ (these terms will not matter in the
computations below) and $Q^i_+(z), Q^i_-(z)$ are relatively prime
polynomials with $Q^i_+(z)$ monic for each $i=1,\ldots,r$. Formula
\eqref{gaugeA4} shows that, without loss of generality, we can and
will assume that each $y_i(z)$ is a monic polynomial.

Acting on the two-dimensional subspace $W_i$ introduced in Section
\ref{rank2}, $v(z)$ has the form
\begin{equation}    \label{vzz}
v(z)|_{W^i}=
\begin{pmatrix}
  y_i(z) & 0\\
  0& y_i^{-1}(z)\prod_{j\neq i} y_j^{-a_{ji}}(z)
 \end{pmatrix}
 \begin{pmatrix}
1 & - \frac{Q^i_{-}(z)}{Q^i_{+}(z)}\\
 0& 1
 \end{pmatrix}
\end{equation}
while $Z$ has the form
\begin{equation}
Z|_{W_i}=\begin{pmatrix}
\zeta_i & 0\\
  0& \zeta_i^{-1}\prod_{j\neq i} \zeta_j^{-a_{ji}}
   \end{pmatrix}.
\end{equation}

We now apply \eqref{2flagformula} and \eqref{gaugeA4}to relate the
$y_i(z)$'s and $Q^i_{\pm}(z)$'s. First, comparing the diagonal entries
on both sides of \eqref{gaugeA4} gives formula \eqref{giyi}:
\begin{equation}    \label{giz}
g_i(z)=\zeta_i\frac{y_i(qz)}{y_i(z)}.
\end{equation}
Second, by comparing the upper triangular entries on both sides of
\eqref{gaugeA4}, we obtain
\begin{multline}    \label{Lambdai}
\Lambda_i(z)\prod_{j>i}g_j(z)^{-a_{ji}} \; = \\ y_i(z)y_i(qz)\prod_{j\neq
  i}y_j(z)^{a_{ji}}\left[
\zeta_i\frac{Q^i_{-}(z)}{Q^i_{+}(z)}-\zeta_i^{-1}\prod_{j\neq i} \zeta_j^{-a_{ji}}
\frac{Q^i_{-}(qz)}{Q^i_{+}(qz)}\right].
\end{multline}
Since $\Lambda_i(z)$ and $y_i(z)$ are monic polynomials, the
nondegeneracy conditions can only be satisfied if
\begin{equation}    \label{yiQi}
y_i(z)=Q_+^i(z), \qquad i=1,\ldots,r.
\end{equation}
Substituting \eqref{yiQi} into \eqref{Lambdai}, we see that the
polynomials $Q^i_+(z), Q^i_-(z)$, $i=1,\ldots,r$, satisfy the
system of equations \eqref{qq}. Thus, we obtain a map from the set of
nondegenerate Miura $(G,q)$-opers to the set of nondegenerate
solutions of \eqref{qq}.

To show that this map is a bijection, we construct its
inverse. Suppose that we are given a nondegenerate solution $\{
Q^i_+(z), Q^i_-(z) \}_{i=1,\ldots,r}$ of the system \eqref{qq}. The
nondegeneracy condition implies that the polynomials $Q^i_+(z)$ and
$Q^i_-(z)$ are relatively prime. We then define $A(z)$ by formula
\eqref{form of A1}, where we set
$$
g_i(z)=\zeta_i \frac{Q^i_+(qz)}{Q^i_+(z)},
$$
 i.e.
\begin{align}    \label{key}
A(z) &=\prod_j\left[ \zeta_j\frac{Q^j_+(qz)}{Q^j_+(z)}
\right]^{\check{\alpha}_j} e^{\frac{\Lambda_j(z) Q^j_+(z)}{\zeta_j
    Q^j_+(qz)}e_i} \\ &= \prod_j
\Big[\zeta_jQ_{+}^j(qz)\Big]^{\check{\alpha}_j}
e^{\frac{\Lambda_j(z)}{\zeta_j Q_+^j(qz)Q_+^j(z)}e_j}
\Big[{Q_{+}^j(z)}\Big]^{-\check{\alpha}_j}\,.
    \label{key1}
\end{align}
We also set
\begin{equation}    \label{prodi}
v(z) = \prod_{i=1}^r y_i(z)^{\check\alpha_i} \prod_{i=1}^r
e^{-\frac{Q^j_{-}(z)}{Q^j_{+}(z)} e_i}.
\end{equation}
Equations \eqref{gaugeA4} are satisfied for all $i=1,\ldots,r$.
Using Proposition \ref{nondeg1}, we check that the nondegeneracy
conditions on $A(z)$ are satisfied. Therefore, $A(z)$ defines a
nondegenerate $Z$-twisted Miura-Pl\"ucker $(G,q)$-oper. This completes
the proof.
\end{proof}

\begin{Rem}
The system \eqref{qq} depends on our choice of ordering of the simple
roots of $G$. In Section \ref{equiv} we will show that the systems
corresponding to different orderings are equivalent.\qed
\end{Rem}

\subsection{Prior work on the $QQ$-system}    \label{prior}

The system \eqref{qq} has an interesting history. As far as we know,
for $G=\SL(2)$ the corresponding equation \eqref{QQ sl2}
with $\Lambda(z)=1$ was first written by Bazhanov, Lukyanov, and
Zamolodchikov \cite{BLZ} in their study of the quantum KdV system. It
was then generalized to the case $G=\SL(3)$ (also with
$\Lambda_i(z)=1$) in \cite{Bazhanov:2001xm}. However, in both of these
works, the conditions imposed on $Q^i_\pm(z)$ are different from those
considered here; they are not polynomials, but rather entire
functions in $z$ with a particular asymptotic behavior as $z
\rightarrow \infty$.

For a general simply laced $G$, the system \eqref{qq} with
$\Lambda_i(z)=1$ is equivalent to a system that, as far as we know,
was first proposed by Masoero, Raimondo, and Valeri in
\cite{Masoero_2016}, in their study of (differential) affine opers
introduced in \cite{FF:kdv}. (For $G=\SL(n)$, a Yangian version of this
system is closely related to the system introduced in \cite{BFLMS};
see Remark 3.4 of \cite{Frenkel:ac}.) The goal of \cite{Masoero_2016}
was to generalize the results of \cite{BLZ} in light of the conjecture
of \cite{FF:kdv} (see also \cite{Frenkel:ac}) linking the spectra of
quantum $\widehat{\fg}$-KdV system and affine $^L\widehat{\fg}$-opers
on $\P^1$ of a special kind. Here, $^L\widehat{\fg}$ is the affine
Kac-Moody algebra that is Langlands dual to $\widehat{\fg}$, i.e., its
Cartan matrix is the transpose of that of $\widehat{\fg}$. If $\fg$
is simply laced, then $^L\widehat{\fg} = \widehat{\fg}$. The authors
of \cite{Masoero_2016} considered the simplest of the
$\widehat{\fg}$-opers proposed in \cite{FF:kdv}, those corresponding
to the ground states of the quantum $\widehat{\fg}$-KdV system, and
associated to each of them a solution of a system equivalent to
\eqref{qq} with $\Lambda_i(z)=1$. (This was subsequently generalized
in \cite{Masoero:2018rel} by Masoero and Raimondo to the
$\widehat{\fg}$-opers conjectured in \cite{FF:kdv} to correspond to
the excited states of the quantum $\widehat{\fg}$-KdV system.)
However, the meaning of this system from the point of view of quantum
integrable systems remained unclear.

The meaning was revealed in \cite{Frenkel:ac}, where it was shown that
the system of \cite{Masoero_2016} is a system of relations in the
Grothendieck ring $K_0(\cO)$ of the category $\cO$ associated to
$U_q\ghat$, which was introduced in \cite{HJ}. (Actually, it is
$U_{q'}\ghat$, where $q'=q^{-2}$, but we will ignore this here.)
Recall that $\cO$ is a category of representations of the Borel
subalgebra $U_q \widehat{\fb}_+$ of $U_q \widehat{\fg}$ (with respect
to the Drinfeld-Jimbo realization) which decompose into a direct sum
of finite-dimensional weight spaces with respect to the
finite-dimensional Cartan subalgebra. There are two sets of
representations from this category corresponding to $Q^i_+(z)$ and
$Q^i_-(z), i=1,\ldots,r$, whose classes in the Grothendieck ring
$K_0(\cO)$ were proved in \cite{Frenkel:ac} to satisfy the relations
of a system equivalent to the $QQ$-system (with $\Lambda_i(z)=1$). The
polynomials $Q^+_i(z), i=1,\ldots,r$, correspond to the classes of the
so-called prefundamental representations of $U_q \widehat{\fb}_+$,
whereas $Q^-_i(z), i=1,\ldots,r$, correspond to the classes of another
less familiar set of representations of $U_q \widehat{\fb}_+$ which
were introduced in \cite{Frenkel:ac}.

Note that in \cite{Frenkel:ac}, $Q_+^i(z)$ was denoted by $Q_i(z)$ and
$Q_-^i(z)$ by $\wt{Q}_i(z)$, and the system \eqref{qq} was called the
$Q\wt{Q}$-system. Here, we call the system \eqref{qq} the $QQ$-{\em
  system}.

According to the results of \cite{Frenkel:ac}, for every quantum
integrable system in which the commutative algebra $K_0(\cO)$ maps to
the algebra of quantum Hamiltonians, we obtain the system \eqref{qq}
for each common set of eigenvalues of the Hamiltonians corresponding
to $Q^i_+(z), Q^i_-(z), i=1,\ldots,r$. Examples of such integrable
systems include the $U_q\widehat\fg \;$ XXZ-type model
(the XXZ model corresponds to $U_q\widehat\sl_2$, as discussed in the
Introduction) and the quantum $\widehat{\fg}$-KdV system. In both
cases, the Hamiltonians corresponding to the $Q^i_+(z)$'s and
$Q^i_-(z)$'s can be expressed as the transfer-matrices associated to
the above representations of $U_q \widehat{\fb}_+$ from the category
$\cO$.

The difference between the two types of systems is reflected in the
difference between the analytic properties of the $Q^i_+(z)$ and
$Q^i_-(z)$. Namely, for $U_q\widehat\fg$ XXZ-type quantum models
with the space of states $V$ being the tensor product of irreducible
finite-dimensional representations, these should be polynomials up to
a universal factor. (This has been proved for $Q^i_+(z)$ in
\cite{Frenkel:2013uda}, and we expect the same to be true for
$Q^i_-(z)$.) These factors should naturally combine to form the
polynomials $\Lambda_i(z)$ appearing in the $QQ$-system \eqref{qq},
which should be equal to products of the Drinfeld polynomials of the
finite-dimensional representations of $U_q\widehat\fg$ appearing as
factors in $V$ (up to multiplicative shifts by powers of $q$). On the
other hand, in quantum KdV systems, the functions $Q^i_+(z)$ and
$Q^i_-(z)$ are expected to be entire functions on the complex plane
\cite{BLZ,Bazhanov:2001xm,Masoero_2016,Masoero:2018rel}.

In the case of $\fg=\sl_2$, the polynomial $Q^1_+(z)$ corresponds to
the eigenvalues of the so-called Baxter operator in the XXZ-type
quantum spin chain. Together with the transfer-matrix of the
two-dimensional evaluation representation of $U_q \widehat\sl_2$, it
obeys the celebrated Baxter $TQ$-relation. This relation was
generalized in \cite{Frenkel:2013uda} from $U_q \widehat\sl_2$ to $U_q
\widehat\g$, where $\fg$ is an arbitrary simple Lie algebra $\fg$,
thereby proving a conjecture of \cite{Frenkel:ls}.

\begin{Rem}
  For simply laced $\fg$, the form of the $QQ$-system \eqref{qq}
  differs slightly from that of \cite{Masoero_2016,Frenkel:ac}. One
  can relate the two by making small notational adjustments. For
  example, in the case of $\fg=\sl_n$ and the standard ordering of the
  simple roots, we obtain
\begin{equation}
\Lambda_i(z)Q^{i+1}_{+}(qz)Q^{i-1}_{+}(z)=
\wt{\xi}_iQ^i_{-}(z)Q^i_{+}(qz)-\xi_iQ^i_{-}(qz)Q^i_{+}(z),
\end{equation}
which is equivalent to 
\begin{equation*}
\Lambda_i(q^{-1/2}z)Q^{i+1}_{+}(q^{1/2}z)Q^{i-1}_{+}(q^{-1/2}z)=
\wt{\xi}_iQ^i_{-}(q^{-1/2}z)Q^i_{+}(q^{1/2}z)-
\xi_iQ^i_{-}(q^{1/2}z)Q^i_{+}(q^{-1/2}z).
\end{equation*}
Upon making the substitution $Q_{\pm}^{i}(z)={\mb Q}_{\pm}^{i}(q^{\frac{N-i}{2}}z)$
and ${\mb \Lambda}_i(z)=\Lambda_i(q^{\frac{N-i-1}{2}}z)$, we obtain a
more symmetric form of the system which was considered in \cite{KSZ}:
\begin{equation*}
{\mb \Lambda}_i(z){\mb Q}^{i+1}_{+}(z){\mb Q}^{i-1}_{+}(z)=
\frac{\zeta_i}{\zeta_{i+1}}{\mb Q}^i_{-}(q^{-1/2}z)
{\mb Q}^i_{+}(q^{1/2}z)-\frac{\zeta_{i+1}}{\zeta_{i}}{\mb
Q}^i_{-}(q^{1/2}z){\mb Q}^i_{+}(q^{-1/2}z).
\end{equation*}
If  we set ${\mb \Lambda}_i(z)=1$, the latter is equivalent to the system
from \cite{Masoero_2016,Frenkel:ac} corresponding to
$U_{q'}\wh\sl_n$ with $q'=q^{-2}$.
\end{Rem}

Now, suppose that $\fg$ is non-simply laced. In this case, the system
\eqref{qq} is {\em different} from the $Q\wt{Q}$-system of
\cite{Masoero_2016_SL} and \cite{Frenkel:ac} corresponding to
$U_q \widehat{\fg}$. Instead, it can be obtained by ``folding'' the
$Q\wt{Q}$-system corresponding to $U_q \widehat{\g'}$, where $\g'$ is
the simply laced Lie algebra with an automorphism $\sigma$ such that
$(\g')^\sigma = \g$. This will be discussed in \cite{FrenkelHern:new}.

\subsection{$QQ$-system and Bethe Ansatz equations}

As we will see, the $QQ$-system \eqref{qq} gives rise to a system of
equations only involving the $Q_+^i(z)$'s. Let $\{ w^k_i
\}_{k=1,\ldots,m_i}$ be the set of roots of the polynomial
$Q^i_+(w)$.  We call the system of equations
\begin{equation}    \label{bethe}
\frac{Q_+^{i}(qw^k_i)}{Q_+^{i}(q^{-1}w^k_i)} \prod_j\zeta_j^{a_{ji}} =
- \; \; \frac{\Lambda_i(w_k^i)\prod_{j>
  i}\Big[Q^j_{+}(qw_k^i)\Big]^{-a_{ji}}\prod_{j<
  i}\Big[Q^j_{+}(w_k^i)\Big]^{-a_{ji}}}{\Lambda_i(q^{-1}w_k^i)\prod_{j>
  i}\Big[Q^j_{+}(w_k^i)\Big]^{-a_{ji}}\prod_{j<
  i}\Big[Q^j_{+}(q^{-1}w_k^i)\Big]^{-a_{ji}}}
\end{equation}
for $i=1,\ldots,r$, $k=1,\ldots,m_i$ the {\em Bethe Ansatz equations}
for the group $G$ and the set $\{ \Lambda_i(z) \}_{i=1,\ldots,r}$.

For simply laced $G$, this system is equivalent to the system of Bethe
Ansatz equations that appears in the $U_q \ghat$ XXZ-type
model~\cite{OGIEVETSKY1986360,RW,Reshetikhin:1987}. However, for
non-simply laced $G$, we obtain a different system of Bethe Ansatz
equations, which,  as far as we know, has not yet been studied in the
literature on quantum
integrable systems. (As we mentioned in the
Introduction, an additive version of this system appeared earlier in
\cite{Mukhin_2005}.) As will be explained in \cite{FrenkelHern:new},
these Bethe Ansatz equations correspond to a novel quantum integrable
model in which the spaces of states are representations of the twisted
quantum affine Kac-Moody algebra $U_q {}^L\ghat$, where $^L\ghat$ is
the Langlands dual Lie algebra of $\ghat$.

Recall the nondegeneracy condition for the solutions of the
$QQ$-system. We apply the same notion to the solutions of
\eqref{bethe}.

\begin{Thm}    \label{BAE}
  There is a bijection between the sets of nondegenerate
  polynomial solutions of the $QQ$-system \eqref{qq} and the Bethe
  Ansatz equations \eqref{bethe}.
\end{Thm}

\begin{proof} Let $\{ Q_+^i(z),Q_-^i(z) \}_{i=1,\ldots,r}$ be a
  nondegenerate solution of the $QQ$-system \eqref{qq}. Set
\begin{equation}    \label{phii}
\phi_i(z)=\frac{Q^i_-(z)}{Q^i_+(z)}
\end{equation}
and
\begin{equation}
f_i(z)=\Lambda_i(z)\prod_{j>
    i}\Big[Q^j_{+}(qz)\Big]^{-a_{ji}}\prod_{j<
    i}\Big[Q^j_{+}(z)\Big]^{-a_{ji}}.
\end{equation}
Then, the $i$th equation of the $QQ$-system may be rewritten as
\begin{equation}    \label{e:gf}
\wt \xi_i\phi_i(z)-\xi_i\phi_i(qz)=\frac{f_i(z)}{Q^i_{+}(z)Q^i_+(qz)}.
\end{equation}
The nondegeneracy condition implies that we have the following partial
fraction decompositions in which all the denominators are pairwise relatively prime:
\begin{align}    \label{hi}
&\frac{f_i(z)}{Q^i_{+}(z)Q^i_+(qz)}=h_i(z)+
\sum_{k=1}^{m_i}\frac{b_k}{z-w^i_{k}}
+ \sum_{k=1}^{m_i} \frac{c_k}{qz-w^i_{k}},\\    \label{tildephi}
&\phi_i(z)=\wt{\phi}_i(z)+\sum_{k=1}^{m_i} \frac{d_k}{z-w^i_{k}}.
\end{align}
Here, $h_i(z)$ and $\wt{\phi}_i(z)$ are polynomials and $\{ w^i_{k}
\}_{k=1,\ldots,m_i}$ is the set of roots of the polynomial
$Q_+^i(z)$.

The residues at $z=w^i_{k}$ (resp. $z=w^i_{k} q^{-1}$) on the two
sides of \eqref{e:gf} must coincide. Therefore,
\begin{equation}    \label{dk}
d_k=\frac{b_k}{\xi_i}, \qquad \operatorname{resp.} \quad
d_k=-\frac{c_k}{\wt{\xi}_i}
\end{equation}
for all $k=1,\ldots,m_i$. Thus, we obtain
$$
\frac{b_k}{\xi_i} + \frac{c_k}{\wt{\xi}_i} = 0,
$$
or equivalently,
\begin{equation}
 \Res_{z=w^i_{k}}\left[\frac{f_i(z)}{\xi_{i}
     Q^i_{+}(z)Q^i_+(qz)}\right]+
 \Res_{z=w^i_{k}}
\left[\frac{f_i(q^{-1}z)}{\wt{\xi}_iQ^i_{+}(q^{-1}z)Q^i_+(z)}\right]=0 
\end{equation}
which is just equation \eqref{bethe} for $i, k$.  Thus, we have a
nondegenerate solution $\{ Q_+^i(z) \}_{i=1,\ldots,r}$ of the system
\eqref{bethe}.

Next, we define the inverse map. Suppose that we have a nondegenerate
solution $\{ Q_+^i(z) \}_{i=1,\ldots,r}$ of \eqref{bethe}. We
need to construct polynomials $\{ Q_-^i(z) \}_{i=1,\ldots,r}$ that
together with the polynomials $\{ Q_+^i(z) \}_{i=1,\ldots,r}$ solve
the $QQ$-system \eqref{qq}. To do this, we will construct a rational
function $\phi_i(z)$ that has the same set of poles as the set of
roots of the polynomial $Q_+^i(z)$ and define $Q_-^i(z)$ by the
formula
$$
Q_-^i(z) = \phi_i(z) Q_+^i(z)
$$
(compare with \eqref{phii}).

We will define the rational function $\phi_i(z)$ via the partial
fraction decomposition \eqref{tildephi}, where we use \eqref{dk} to
define the residues $d_k$.  (Note that the $d_k$'s are completely
determined by the $Q_+^i(z)$'s.) It remains to find the polynomial part
of $\phi_i(z)$,
$$
\wt{\phi}_i(z)=\sum_{m\geq 0} r_mz^m.
$$

Let
$$
h_i(z)=\sum_{m\geq 0} s_m z^m
$$
be the polynomial appearing in equation \eqref{hi}. Note that
$h_i(z)$, and hence $\{ s_m \}_{m\geq 0}$, are completely determined
by the $Q_+^i(z)$'s.

Now, observe that \eqref{e:gf}, which is the $i$th equation of the
$QQ$-system, is satisfied if and only if the following equations on
the $\{ r_m \}_{m\geq 0}$ are satisfied:
\begin{equation}    \label{rm}
r_m(\wt{\xi}_iq^m -{\xi}_i )=s_m, \qquad m \geq 0.
\end{equation}
This follows from our assumption on $Z$ in \eqref{assume} because
$\wt\xi_i/\xi_i=\prod_{j=1}^r \zeta_j^{a_{ji}}$. Therefore, each of
the equations \eqref{rm} has a unique solution.

It then follows that there exist unique polynomials $\{ {Q}^i_-(z)
\}_{i=1,\ldots,r}$ that together with $\{ {Q}^i_+(z)
\}_{i=1,\ldots,r}$ satisfy the $QQ$-system \eqref{qq}. Furthermore, by
construction, it follows that this solution of the $QQ$-system is
nondegenerate.
\end{proof}

\section{B\"acklund-type transformations}    \label{Sec:Backlund}

Theorems \ref{inj} and \ref{BAE} establish a bijection between the set
of nondegenerate $Z$-twisted Miura-Pl\"ucker $(G,q)$-opers and the
sets of polynomial nondegenerate solutions of the $QQ$-system and the
Bethe Ansatz equations \eqref{bethe}.

Now, the set of nondegenerate $Z$-twisted Miura-Pl\"ucker
$(G,q)$-opers includes as a subset those $Z$-twisted Miura-Pl\"ucker
$(G,q)$-opers which are actually $Z$-twisted Miura
$(G,q)$-opers. Recall the difference between the two: a $Z$-twisted
Miura $(G,q)$-oper is one whose $q$-connection can be represented in
the form \eqref{gaugeA2}:
\begin{equation}    \label{gaugeA5}
A(z)=v(qz)Z v(z)^{-1}, \qquad v(z) \in B_+(z),
\end{equation}
whereas a $Z$-twisted Miura-Pl\"ucker $(G,q)$-oper is one for which
only the associated $(\GL(2),q)$-opers $A_i(z)$ have this property
(compare with \eqref{gaugeA4}). When we constructed the
inverse map in the proof of Theorem \ref{inj}, we defined an element
$v(z)$ of $B_+(z)$
by formula \eqref{prodi}. This $v(z)$ satisfies the equations
\eqref{gaugeA4}, so we do get a $Z$-twisted Miura-Pl\"ucker
$(G,q)$-oper, but it is not clear whether this $v(z)$ can be extended
to an element of $B_+(z)$ satisfying formula \eqref{gaugeA5}. More
precisely, equations \eqref{gaugeA4} uniquely fix the image
$\overline{v}(z)$ of $v(z)$ in the quotient $B_+/[N_+,N_+]$, and the
question is whether we can lift this $\overline{v}(z)$ to an element
$v(z) \in B_+(z)$ such that equation \eqref{gaugeA5} is satisfied.

In this section, we will give a sufficient condition for this to hold
(see Theorem \ref{w0} and Remark \ref{3sets}). It is based on
transformations described in the next subsection for generating new
solutions of the $QQ$-system from an existing one. (There is one such
transformation for each simple root of $G$.). We call them
B\"acklund-type transformations.

Here, we follow an idea of Mukhin and Varchenko
\cite{Mukhin_2003,Mukhin_2005}, who introduced similar procedures for
the solutions of the Bethe Ansatz equations arising from the XXX-type
models associated to Yangians. However, in contrast to their setting,
we have a non-trivial twist represented by a regular semisimple
element $Z$ of the Cartan subalgebra. As a result, our transformations
generically give rise to solutions labeled by elements of the Weyl
group of $G$, rather than by points of the flag manifold of $G$ as in
\cite{Mukhin_2003,Mukhin_2005}.

\subsection{Definition of B\"acklund-type
  transformations}    \label{backlund}

Consider a $Z$-twisted Miura-Pl\"ucker $(G,q)$-oper given by formula
\eqref{key}. We now define a transformation associated to the $i$th
simple reflection from the Weyl group $W_G$ on the set of such Miura
$q$-opers.

\begin{Prop}    \label{fiter}
  Consider the $q$-gauge transformation of the $q$-connection $A$
  given by formula \eqref{key}:
\begin{eqnarray}
A \mapsto A^{(i)}=e^{\mu_i(qz)f_i}A(z)e^{-\mu_i(z)f_i},
\quad \operatorname{where} \quad \mu_i(z)=\frac{\prod\limits_{j\neq
    i}\Big[Q_+^j(z)\Big]^{-a_{ji}}}{Q^i_{+}(z)Q^i_{-}(z)}\,.
\label{eq:PropDef}
\end{eqnarray}
Then $A^{(i)}(z)$ can be obtained from $A(z)$ by
substituting in formula \eqref{key} (or \eqref{key1})
\begin{align}
Q^j_+(z) &\mapsto Q^j_+(z), \qquad j \neq i, \\
Q^i_+(z) &\mapsto Q^i_-(z), \qquad Z\mapsto s_i(Z)\,.
\label{eq:Aconnswapped}
\end{align}
\end{Prop}

In the proof of Theorem \ref{fiter}, we will use the following lemma,
which is proved by a direct computation. (The results of the lemma
have appeared previously in \cite{Mukhin_2005}).

\begin{Lem}
The following relations hold for any $u,v\in\mathbb{C}$:
\begin{align}
u^{\check{\alpha}_i} e^{v e_j} &=
\exp\left({u^{a_{ji}}v\, e_i} \right)u^{\check{\alpha}_i}\,\notag \\
u^{\check{\alpha}_i} e^{v f_j} &=
\exp\left({u^{-a_{ji}}v\, f_i} \right)u^{\check{\alpha}_i}\,\notag \\
e^{ue_i}~e^{vf_i}&=\exp\left({\frac{v}{1+uv}
f_i}\right)~(1+uv)^{\check{\alpha}_i}~\exp\left({\frac{u}{1+uv}
e_i}\right)\,.\notag
\end{align}
\label{Th:LemmaMV05}
\end{Lem}

\medskip

\noindent{\em Proof of Proposition \ref{fiter}}.
  Using the first identity from Lemma \ref{Th:LemmaMV05}, we can move
  all factors $Q^j_+(qz)$ in formula \eqref{key} to the left and all
  factors $Q^j_+(z)$ to the right. The resulting expression is
\begin{eqnarray}
A(z)=
\prod_k \Big[Q_{+}^k(qz)\Big]^{\check{\alpha}_k}
\left[\prod_i\zeta_i^{\check{\alpha}_i}e^{\wt{\Lambda}_i(z)e_i}\right]\prod_l
\left[{Q_{+}^l(z)}\right]^{-\check{\alpha}_l},
\label{eq:NewAtilde}
\end{eqnarray}
where 
\begin{equation}
\wt{\Lambda}_i(z)=\frac{\prod\limits_{j<i}\Big[Q_+^j(z)\Big]^{-a_{ji}}\prod\limits_{j>i}
  \Big[Q_+^j(qz)\Big]^{-a_{ji}}\Lambda_i(z)}{\zeta_i Q_+^i(qz)Q_+^i(z)}\,.
\label{eq:DefLambdaTilde}
\end{equation}
Let
\begin{equation}
\wt{\mu}_i(z)=\mu_i(z)\prod_j\Big[Q_+^j(z)\Big]^{a_{ji}}\,.
\label{eq:Mutdef}
\end{equation}
Then, applying the second identity from Lemma \ref{Th:LemmaMV05} to
\eqref{eq:NewAtilde}, we obtain
\begin{equation}\label{insprod}
\widetilde A^{(i)}=e^{\mu_i(qz)f_i}A(z)e^{-\mu_i(z)f_i}=\dots
\zeta_i^{\check{\alpha}_i} e^{w\, f_i}\cdot e^{u\, e_i}\cdot e^{v\,f_i}\dots
\end{equation}
where 
\begin{equation}
w=\zeta^2_i\prod_{j<i}\zeta_j^{a_{ji}}\wt{\mu}_i(qz)\,,\qquad u=\wt{\Lambda}_i(z)\,,\qquad v=-\prod_{j>i}\zeta_j^{a_{ji}}\wt{\mu}_i(z)\,,
\end{equation}
and the ellipses stand for all other terms including the elements of
the maximal torus and the exponentials of $e_j$ with $j\neq i$. We now
use the third identity from Lemma \ref{Th:LemmaMV05} to reshuffle the
middle and the last exponent in \eqref{insprod}:
\begin{equation}
\widetilde A^{(i)} = \dots \zeta_i^{\check{\alpha}_i}\,
\exp\left({\left(w+\frac{v}{1+uv}\right)
    f_i}\right)~(1+uv)^{\check{\alpha}_i}~\exp\left({\frac{u}{1+uv}
    e_i}\right)\dots\,.
\label{eq:Anewform}
\end{equation}

In order to prove the proposition, we first need to show that
\begin{equation}
w+\frac{v}{1+uv}=0, 
\end{equation}
or, in other words, 
\begin{equation}
\zeta^2_i\prod_{j<i}\zeta_j^{a_{ji}}\wt{\mu}_i(qz)-
\frac{\prod_{j>i}\zeta_j^{a_{ji}}\wt{\mu}_i(z)}{1-\wt{\Lambda}_i(z)
\prod_{j>i}\zeta_j^{a_{ji}}\wt{\mu}_i(z)}=0.
\end{equation}
 Let us demonstrate that this is indeed the case. The above equation is equivalent to
\begin{equation}
\prod_{j>i}\zeta_j^{-a_{ji}}\wt{\mu}^{-1}_i(z)-\zeta_i^{-2}\prod_{j<i}\zeta_j^{-a_{ji}}\wt{\mu}^{-1}_i(qz)=\wt{\Lambda}_i(z)\,.
\end{equation}
Substituting $\wt{\Lambda}_i(z)$ from \eqref{eq:DefLambdaTilde} into this equation gives 
\begin{align}
\zeta_i\prod_{j>i}\zeta_j^{-a_{ji}}\,\wt{\mu}^{-1}_i(z)Q_+^i(qz)Q_+^i(z)&-\zeta_i^{-1}\prod_{j<i}\zeta_j^{-a_{ji}}\,\wt{\mu}^{-1}_i(qz)Q_+^i(qz)Q_+^i(z)\notag\\
&=\prod_{j<i}\Big[Q_+^j(z)\Big]^{-a_{ji}}\prod_{j>i}\Big[Q^j_+(qz)\Big]^{-a_{ji}}\Lambda_i(z)\,.
\end{align}
Keeping in mind from \eqref{eq:PropDef} and \eqref{eq:Mutdef} that
$\wt{\mu}_i(z)=\frac{Q^i_{+}(z)}{Q^i_{-}(z)}$, we recover the $QQ$-system equations:
\begin{align}
\zeta_i\prod_{j>i}\zeta_j^{-a_{ji}}\,Q_+^i(qz)Q_-^i(z)&-\zeta_i^{-1}\prod_{j<i}\zeta_j^{-a_{ji}}\,Q_-^i(qz)Q_+^i(z)\notag\\
&=\prod_{j<i}\Big[Q_+^j(z)\Big]^{-a_{ji}}\prod_{j>i}\Big[Q^j_+(qz)\Big]^{-a_{ji}}\Lambda_i(z)\,.
\end{align}

We now check that the two remaining terms in \eqref{eq:Anewform} are indeed as prescribed by \eqref{eq:Aconnswapped}. Note that 
\begin{eqnarray}
&&1+uv=1-\wt{\Lambda}_i(z)\prod_{j>i}\zeta_j^{a_{ji}}\wt{\mu}_i(z)\nonumber \\
&&=1-\frac{Q^i_{+}(z)}{Q^i_{-}(z)}\prod_{j>i}\zeta_j^{a_{ji}}\frac{\prod_{j<i}\Big[Q_+^j(z)\Big]^{-a_{ji}}\prod_{j>i}\Big[Q^j_+(qz)\Big]^{-a_{ji}}\Lambda_i(z)}{\zeta_i Q_+^i(qz)Q_+^i(z)}\,.
\end{eqnarray}
Thus, the $QQ$-equations imply that 
\begin{equation}
1+uv=\zeta_i^{-2}\prod_{j\neq i}\zeta_j^{-a_{ji}}\frac{Q_-^i(qz)Q_+^i(z)}{Q_+^i(qz)Q_-^i(z)}\,.
\end{equation}
The corresponding contribution to the q-connection is
\begin{eqnarray}
(1+uv)^{\check{\alpha}_i}=
s_i(Z)Z^{-1}\left[\frac{Q_-^i(qz)}{Q_+^i(qz)}\right]^{\check{\alpha}_i}\left[\frac{Q_+^i(z)}{Q_-^i(z)}\right]^{\check{\alpha}_i}\,.
\end{eqnarray}
Finally, we need to evaluate the last term in
\eqref{eq:Anewform}. Computing, we obtain 
\begin{equation}\begin{aligned}
  \frac{u}{1+uv}&=\frac{\displaystyle\frac{\prod_{j<i}\Big[Q_+^j(z)\Big]^{-a_{ji}}\prod_{j>i}\Big[Q^j_+(qz)\Big]^{-a_{ji}}\Lambda_i(z)}{\zeta_i Q_+^i(qz)Q_+^i(z)}}{\zeta_i^{-2}\prod_{j\neq i}\zeta_j^{-a_{ji}}\frac{\displaystyle Q_-^i(qz)Q_+^i(z)}{\displaystyle Q_+^i(qz)Q_-^i(z)}}\nonumber \\
  &=\frac{\prod_{j<i}\Big[Q_+^j(z)\Big]^{-a_{ji}}\prod_{j>i}\Big[Q^j_+(qz)\Big]^{-a_{ji}}\Lambda_i(z)}{\zeta_i^{-1}\prod_{j\neq i}\zeta_j^{-a_{ji}}Q_-^i(qz)Q_-^i(z)}\left[\frac{Q_-^i(z)}{Q_+^i(z)}\right]^2\,.
\end{aligned}
\end{equation}
Thus, from \eqref{eq:Anewform}, we have
\begin{equation}
\widetilde A^{(i)} =\dots
\zeta_i^{-\check\alpha_i-\langle\check\alpha_j, \alpha_i\rangle\check\alpha_i}\left[\frac{Q_-^i(qz)}{Q_+^i(qz)}\right]^{\check{\alpha}_i}
e^{\wt{\Lambda}^{-}_i(z) e_i}
\left[\frac{Q_+^i(z)}{Q_-^i(z)}\right]^{\check{\alpha}_i}\dots\,,
\end{equation}
where 
$$\wt\Lambda^{-}_i(z)=\frac{\prod_{j<i}\Big[Q_+^j(z)\Big]^{-a_{ji}}\prod_{j>i}\Big[Q^j_+(qz)\Big]^{-a_{ji}}\Lambda_i(z)}{\zeta_i^{-1}\prod\limits_{j\neq
    i}\zeta_j^{-a_{ji}}Q_-^i(qz)Q_-^i(z)}\,.$$ Moving the elements of
the maximal torus corresponding to the ratios of the $Q$'s to the left
and right side of \eqref{insprod}, we obtain the statement of
Proposition \ref{fiter}.\qed

\subsection{Nondegeneracy}

Suppose that $A(z)$ is a nondegenerate $Z$-twisted Miura-Pl\"ucker
$(G,q)$-oper. In Proposition \ref{fiter}, we apply to it the $q$-gauge
transformation by $e^{\mu_i(z)f_i}$, which belongs to $N_-$ and not to
$B_+$. Hence, it is not clear whether the $q$-connection $A^{(i)}(z)$
constructed in Proposition \ref{fiter} corresponds to a $Z$-twisted
Miura-Pl\"ucker $(G,q)$-oper, let alone a nondegenerate one.

Seen from the perspective of Theorem \ref{inj}, there exist unique
polynomials $\{ Q^j_-(z) \}_{j=1,\ldots,r}$ that together with $\{
Q^j_+(z) \}_{j=1,\ldots,r}$ give rise to a nondegenerate solution of
the $QQ$-system \eqref{qq} corresponding to the element $Z \in H$.
However, the nondegeneracy condition for the latter solution does {\em
  not} include any information on the roots of the polynomials $\{
Q^j_-(z) \}_{j=1,\ldots,r}$, other than the fact that the roots of
$Q^j_-(z)$ are distinct from the roots of $Q^j_+(z)$ for all
$j=1,\ldots,r$.  (We know this from the construction of $Q^j_-(z)$
given in the proof of Theorem \ref{BAE}).

By construction, $A^{(i)}(z)$ is an $s_i(Z)$-twisted Miura-Pl\"ucker
$(G,q)$-oper, which corresponds to the polynomials $\{ \wt{Q}^j_+(z)
\}_{j=1,\ldots,r}$, where $\wt{Q}^+_j(z)=Q^+_j(z)$ for $j \neq i$ and
$\wt{Q}^+_i(z) = Q^-_i(z)$. The conditions for it to be nondegenerate
are spelled out in the following lemma.

\begin{Lem} \label{nondegcond} Suppose that the roots of the
  polynomial $Q^i_-(z)$ constructed in the proof of Theorem \ref{BAE}
  are $q$-distinct from the roots of $\Lambda_k(z)$ for $a_{ik} \neq
  0$ and from the roots of $Q^j_+(z)$ for $j\neq i$ and $a_{jk} \neq
  0$.  Then, the data
\begin{align} \label{qqm}
  \{ \wt{Q}^j_+ \}_{j=1,\ldots,r} &= \{ Q^1_{+}, \dots,
  Q^{i-1}_+,Q^i_-,Q^{i+1}_+ \dots , Q^r_{+} \}; \\ \notag
  \{ \wt{\zeta}_j \}_{j=1,\ldots,r} &= \{
  \zeta_1,\dots,\zeta_{i-1},\zeta_i^{-1}\prod\limits_{j\neq
  i}\zeta_j^{-a_{ji}},\dots,\zeta_r\}
\end{align}
give rise to a nondegenerate solution of the Bethe Ansatz equations
\eqref{bethe} corresponding to $s_i(Z) \in H$.  Furthermore, there
exist polynomials $\{ \wt{Q}^j_- \}_{j=1,\ldots,r}$ that together with
$\{ \wt{Q}^j_+ \}_{j=1,\ldots,r}$ give rise to a nondegenerate
solution of the $QQ$-system \eqref{qq} corresponding to $s_i(Z)$.
\end{Lem}

\begin{proof}
Let us rewrite the $QQ$-system as follows: 
\begin{equation}\label{qqgen}
\Lambda_i(z)\prod_{j}\Big[Q^j_{+}(q^{b_{ji}}z)\Big]^{-a_{ji}}=
\wt{\xi}_iQ^i_{-}(z)Q^i_{+}(qz)-\xi_iQ^i_{-}(qz)Q^i_{+}(z),
\end{equation}
where $b_{ji}=1$ if $j>i$ and $b_{ji}=0$ otherwise. Note that for
$i\ne j$,
\begin{equation}
b_{ji}+b_{ij}=1\,.
\label{eq:bijbji1}
\end{equation}

Dividing both sides by $Q^i_{+}(z)Q^i_{-}(z)$ gives
\begin{equation}
\wt{\xi}_i\frac{Q^i_{+}(qz)}{Q^i_{+}(z)}-\xi_i
\frac{Q^i_{-}(qz)}{Q^i_{-}(z)}=\frac{\Pi_i(z)}{Q^i_{+}(z)Q^i_{-}(z)}\,,
\end{equation}
where
$$
\Pi_i(z)=\Lambda_i(z)\prod_{j\neq
  i}\left[Q^j_+(q^{b_{im}}z)\right]^{-a_{ji}}.
$$
Evaluating this equation at $q^{-b_{im}}w^k_i$, where $\{
w^k_m \}_{i=1,\ldots,m_i}$ is the set of roots of the polynomial
$Q_+^i(z)$, we obtain
\begin{equation}
\prod_{j\neq
  i}\zeta_j^{a_{ji}}\frac{Q^j_{+}(q^{1-b_{im}}w^k_m)}{Q^j_{+}(q^{-b_{im}}w^k_m)}
=\frac{Q^i_{-}(q^{1-b_{im}}w^k_m)}{Q^i_{-}(q^{-b_{im}}w^k_m)}\,.
\end{equation}
Comparing the above formula with the Bethe Ansatz equations
\eqref{bethe} and using \eqref{eq:bijbji1}, we see that the data
(\ref{qqm}) satisfy the Bethe Ansatz equations \eqref{bethe}.  The
existence of $\{ \wt{Q}^j_+ \}_{j=1,\ldots,r}$ then follows from
Theorem \ref{BAE}.
\end{proof}

Thus, if the conditions of Lemma \ref{nondegcond} are
satisfied, we can associate to every nondegenerate $Z$-twisted
Miura-Pl\"ucker $q$-oper a nondegenerate $s_i(Z)$-twisted
Miura-Pl\"ucker oper via $A(z) \mapsto A^{(i)}(z)$. We call this
procedure a {\em B\"acklund-type transformation} associated to the
$i$th simple reflection of the Weyl group $W_G$. We now
generalize this transformation to other Weyl group elements.

\begin{Def}
  Let $w=s_{i_1} \dots s_{i_k}$ be a reduced decomposition of an
  element $w$ of the Weyl group. A solution of the $QQ$-system
  \eqref{qq} is called $({i_1} \dots {i_k})$-generic if by
  consecutively applying the procedure described in Lemma
  \ref{nondegcond} with $i = i_k, \dots,i_1$, we obtain a sequence
  of nondegenerate solutions of the $QQ$-systems corresponding to the
  elements $w_j(Z) \in H$, where $w_k=s_{i_{k-j+1}} \ldots s_{i_k}$
  with $j=1,\ldots,k$.
\end{Def}

\begin{Rem}
  Note that in this definition, we only assume the existence of a
  sequence of transformations as described in Lemma \ref{nondegcond}
  for a particular reduced decomposition of $w$. We do not assume that
  such a sequence exists for other reduced decompositions of $w$.
\end{Rem}

Now, we derive an important consequence of this property for
$(G,q)$-opers.

\begin{Prop} \label{bminus} Let $w=s_{i_1} \dots s_{i_k}$ be a reduced
  decomposition.  Then, for each $({i_1} \dots {i_k})$-generic
  solution of the $QQ$-system \eqref{qq}, there exists an element
  $b_-(z)\in B(z)$ of the form
$$
b_-(z)=e^{c_{i_1}(z)f_{i_1}}e^{c_{i_2}(z)f_{i_2}}\dots
  e^{c_{i_k}(z)f_{i_k}}h(z),
$$
where $c_{i_j}(z)$ are non-zero rational functions and $h(z)\in H(z)$,
such that
\begin{equation}
b_-(qz)w(Z)v=A(z)b_-(z)v.
\label{eq:OperatorDiffbm}
\end{equation}
Here, $A(z)$ is given by equation \eqref{key}
and $v$ is a highest weight vector in any irreducible
finite-dimensional representation of $G$.
\end{Prop}

\begin{proof} 
  The idea is to construct $b_-(z)$ as a composition of the elements
  of $B_-(z)$ appearing in the $q$-gauge transformations from the
  Proposition \ref{fiter} corresponding to different simple
  reflections in the reduced decomposition of $w$.

  If $w=s_i$, then we have the single $q$-gauge transformation from
  formula \eqref{eq:PropDef}:
\begin{equation}\label{gaugefirst}
A\mapsto A^{(i)}=e^{\mu_i(qz)f_i}A(z)e^{-\mu_i(z)f_i},
\quad \operatorname{where} \quad \mu_i=\frac{\prod\limits_{j\neq
    i}\Big[Q_+^j(z)\Big]^{-a_{ji}}}{Q^i_{+}(z)Q^i_{-}(z)}\,.
\end{equation}
According to Proposition \ref{fiter}, if we apply $A^{(i)}$ to a
highest weight vector $v$, we obtain
\begin{equation}
A^{(i)}v=\prod_j\Bigg[\frac{\wt{Q}^j_+(qz)}{\wt{Q}^j_+(z)}
\Bigg]^{\check{\alpha}_j}s_i(Z)~v,
\end{equation}
where the polynomials $\wt Q^j_+(z)$ are defined by formula
\eqref{qqm}. After multiplying both sides by
$$
e^{-\mu_i(qz)f_i}\prod_j
\Big[\wt{Q}^j_+(z)\Big]^{\langle\check\alpha_j,\lambda\rangle},
$$
where $\lambda$ is the weight of $v$, we obtain:
\begin{equation}
A(z)e^{-\mu_i(z)f_i}\prod_j
\Big[\wt{Q}^j_+(z)\Big]^{\check{\alpha}_j}v=e^{-\mu_i(qz)f_i}\prod_j
\Big[\wt{Q}^j_+(qz)\Big]^{\check{\alpha}_j}s_i(Z)v.
\end{equation}
Thus, $b_-(z)=e^{-\mu_i(z)f_i}\prod_j
\Big[\wt{Q}^j_+(z)\Big]^{\check{\alpha}_j}$ is the sought-after element
of $B_-(z)$. 

For $w=s_{i_1} \dots s_{i_k}$, we successively apply the $q$-gauge
transformations \eqref{gaugefirst} with $i=i_k, i_{k-1}, \dots, i_1$.
Let $\mu_{i_k}(z),\mu_{i_{k-1}}(z), \dots, \mu_{i_1}(z)$ be the
corresponding functions as in \eqref{eq:PropDef}, and let
$\{\overline{Q}^j_+(z)\}$ be the set of polynomials obtained from
$\{Q^j_+(z)\}$ by successively applying formula \eqref{qqm} with
$i=i_k, i_{k-1}, \dots, i_1$. It then follows that
\begin{equation}
b_-(z)=e^{-\mu_{i_1}(z)f_{i_1}} e^{-\mu_{i_2}(z)f_{i_2}}\dots
e^{-\mu_{i_k}(z)f_{i_k}}\prod_j
\Big[\overline{Q}^j_+(z)\Big]^{\check{\alpha}_j}
\end{equation}
satisfies \eqref{eq:OperatorDiffbm}.
\end{proof}

We also formulate an additional result regarding elements in Bruhat
cells. Recall the following well-known fact about the product of
Bruhat cells (see e.g. \cite[Lemma 29.3.A]{humphreys}):

\begin{Lem} \label{bmult}
If $u, v\in W_G$ satisfy $\ell(u)+\ell(v)=\ell(uv)$, then
$B_-uB_-vB_-=B_-uvB_-$.
\end{Lem}

The following result is a direct consequence of Lemma \ref{bmult}.

\begin{Prop}    \label{red}
If $w\in W$ has a reduced decomposition $w=s_{i_1}s_{i_2}\dots
s_{i_k}$, then $$e^{a_{i_1}e_{i_1}}
e^{a_{i_2}e_{i_2}}\dots e^{a_{i_k}e_{i_k}}\in B_-wB_-, \quad e^{a_{i_1}f_{i_1}}
e^{a_{i_2}f_{i_2}}\dots e^{a_{i_k}f_{i_k}}\in B_+wN_+$$ if
$a_{i_j}\neq 0$ for all $j$.
\end{Prop}

\begin{Def}
  A $Z$-twisted Miura-Pl\"ucker $(G,q)$-oper is called $({i_1} \dots
  {i_k})$-generic if it corresponds to an $({i_1} \dots
  {i_k})$-generic solution of the $QQ$-system via the bijection in
  Theorem \ref{inj}.
\end{Def}

\subsection{From Miura-Pl\"ucker to Miura $q$-opers}    \label{main thm}

We shall now describe a sufficient condition for a
$Z$-twisted Miura-Pl\"ucker $(G,q)$-opers to be a Miura $(G,q)$-oper.

Let $w_0=s_{i_1}\dots s_{i_{\ell}}$ be a reduced decomposition of the
longest element of the Weyl group. In what follows, we refer to an
$(i_1,\dots,i_{\ell})$-generic object as $w_0$-{\em generic}.

\begin{Thm}    \label{w0}
Every $w_0$-generic $Z$-twisted Miura-Pl\"ucker $(G,q)$-oper is a
nondegenerate $Z$-twisted Miura $(G,q)$-oper.
\end{Thm}

\begin{proof} 
Let  $$
A(z)=\prod_j\Bigg[\zeta_j\frac{Q^j_+(qz)}{Q^j_+(z)}\Bigg]^{\check{\alpha}_j}
e^{\frac{\Lambda_j(z)e_j}{g_j(z)}}
$$ be the $w_0$-generic $Z$-twisted Miura-Pl\"ucker $(G,q)$-oper
coming from a $w_0$-generic solution $\{Q^j_{+}\}$ of the $QQ$-system.
By Proposition \ref{bminus}, there exists an element $b_-(z)\in
B_-(z)$ such that 
$$
b_-(qz)w_0(Z)v=A(z)b_-(z)v,
$$
where $v$ is any highest weight vector in a finite-dimensional
irreducible representation of $G$; moreover, $$
b_-(z)=e^{c_{i_1}f_{i_1}}e^{c_{i_2}f_{i_2}}\dots e^{c_{i_k}f_{i_\ell}}h(z)
$$
with $c_{i_j}(z) \in \C(z)^\times$ and $h(z)\in H(z)$.

By Proposition \ref{red},
$$
b_-(z)=b_{+}(z)w_0 n_+(z),
$$
where $b_+(z)\in B_+(z)$ and $n_+(z)\in N_+(z)$. Therefore, we have
$$
b_+(qz)Zw_0v=A(z)b_+(z)w_0 v,
$$
so if we set
\begin{equation}    \label{Uz}
U(z)=Z^{-1}b^{-1}_+(qz)A(z)b_+(z)\in B_+(z),
\end{equation}
then
$$
w_0v=U(z)w_0 v
$$
for any irreducible finite-dimensional representation of $G$ with
highest weight vector $v$. Thus, $U(z)$ is an element of $B_+(z)$
which fixes the lowest weight vector $w_0 v$ of any irreducible
finite-dimensional representation of $G$. This means that $U(z)=1$.
Equation \eqref{Uz} then implies that $A(z)$ satisfies
\begin{equation}    \label{Abplus}
A(z) = b_+(qz) Z b_+(z)^{-1}
\end{equation}
for some $b_{+}(z) \in B_+(z)$.  Thus, we have proved that every
$w_0$-generic $Z$-twisted Miura-Pl\"ucker $(G,q)$-oper is a
nondegenerate $Z$-twisted Miura $(G,q)$-oper. Equivalently, every
$w_0$-generic solution of the $QQ$-system corresponds to a
nondegenerate $Z$-twisted Miura $(G,q)$-oper.
\end{proof}

\begin{Rem}    \label{3sets}
  Given a regular semisimple element $Z \in H$ and a collection of
  polynomials $\{ \Lambda_i(z) \}_{i=1,\ldots,r}$ as above, consider the
  following three sets of objects on $\P^1$:

\begin{itemize}

\item $q\on{-MPOp}_G^Z$, the set of nondegenerate $Z$-twisted
  Miura-Pl\"ucker $(G,q)$-opers;

\item $q\on{-MPOp}_G^{Z,w_0}$, the set of $w_0$-generic $Z$-twisted
  Miura-Pl\"ucker $(G,q)$-opers;

\item $q\on{-MOp}_G^Z$, the set of nondegenerate $Z$-twisted
  Miura $(G,q)$-opers.

\end{itemize}

According to the discussion at the beginning of this section and
Theorem \ref{w0}, we have the inclusions
\begin{equation}    \label{inclusions}
q\on{-MPOp}_G^{Z,w_0} \; \subset \; q\on{-MOp}_G^Z \; \subset \;
q\on{-MPOp}_G^Z.
\end{equation}
It would be interesting to find out under what conditions either (or
both) of these inclusions is an equality.\qed
\end{Rem}

\subsection{Equivalence of the $QQ$-Systems}    \label{equiv}

The $QQ$-system \eqref{qq} depends on the choice of the Coxeter
element $c$ that we used in the definition of $(G,q)$-opers; the
ordering of the simple reflections in it is the ordering of the simple
roots in \eqref{qq}. In this subsection, we show that $QQ$-systems
corresponding to different Coxeter elements (or equivalently,
different orderings of the simple roots) are equivalent. First, we
explain what we mean by equivalence.

We say that two $QQ$-systems are \emph{gauge equivalent} if they are
related to each other via the transformations
\begin{align}
&Q^i_{\pm}(x)\mapsto Q^i_{\pm}(D^{\{i\}}x)\,,\qquad \Lambda^i(x)\mapsto
\Lambda^i(D^{\{i\}}x)\,, \notag\\
&\xi_i \mapsto \alpha^{(i)}\xi_i\,,\qquad \qquad
\widetilde{\xi}_i\mapsto \alpha^{(i)}\widetilde{\xi}_i\,,
\end{align}
where the $D^{\{i\}}$'s and $\alpha^{(i)}$'s are non-zero complex
numbers.

The following theorem is a multiplicative version of Theorem 2.5 and
Corollary 2.6 in \cite{Mukhin_2005}.

\begin{Thm}
 The $QQ$-system associated to the Coxeter element $c$, parameters
  $\{\zeta_i\}$ and polynomials $\{\Lambda_i(z)\}$ is gauge equivalent
  to the $QQ$-system associated to a Coxeter element $c'$ with the
  same parameters $\{\zeta_i\}$ and polynomials
  $\{\Lambda_i(q^{d^{\{i\}}}z)\}$ for some $d^{\{i\}} \in \Z$.
\end{Thm}

\begin{proof}
  Recall from \eqref{qqgen} that the $QQ$-system defined using the
  Coxeter element $c=w_{i_1}\dots w_{i_r}$ may be written in terms of
  integers $\{b_{ij}\}_{i\ne j}$, where $b_{ij}$ equals $1$ (resp.
  $0$) if $w_j$ comes before (resp. after) $w_i$ in this order.  To
  prove the theorem, it suffices to show that this system is gauge
  equivalent to one defined in terms of other parameters
  $\hat{b}_{ij}$ whose definition only involves the Dynkin diagram.

Since $G$ is simple, the underlying graph of the Dynkin diagram is a
tree whose vertices are labeled by the simple reflections $\a_i$.  We let
$\delta_i$ denote the distance from $\a_i$ to $\a_1$.
  
 %  It suffices to show that any $QQ$-system is gauge equivalent Let us
%   rewrite the $QQ$-system as in formula \eqref{qqgen}, where
%   $b_{ji}=1$ if $j>i$ and $b_{ji}=0$ otherwise. Recall that
% \begin{equation}
% b_{ji}+b_{ij}=1\,.
% \label{eq:bijbji}
% \end{equation}

% Given that $b_{ij}$ satisfy the above property, we will show that each
% such system (\ref{qqgen}) of functional relations is gauge equivalent
% to the one described below. This will immediately lead to the
% statement of the theorem.  First, let us assume that we work with the
% Lie algebras of ADE-type, i.e. Dynkin diagram is simply laced. Assume
% that $s=s_1,\dots, s_r$, so that $s_i$ correspond to the simple root
% $\alpha_i$ $v_1,\dots, v_r$ are the Dynkin diagram vertices.

% For $ i = 2,\ldots, r$, let $v_{i_1}, \ldots , v_{i_k}$ be the unique
% sequence of distinct vertices of the Dynkin diagram such that for
% $j=1,\dots , k-1$ the vertices $v_{i_j}$ and $v_{i_{j+1}}$ are
% connected by an edge, and $v_{i_1}$ = $v_1$, $v_{i_k}$ = $v_i$.  We
% denote $\delta_i:=k$ and call it a distance between $v_1$ and $v_i$.
We now define $\hat{b}_{ij}$ via
\begin{equation}\hat{b}_{ij}=\begin{cases} 1 & \text{ if
      $\delta_i <\delta_j$, or $\delta_i =\delta_j$ and $i<j$};\\
0 & \text{ otherwise.}
\end{cases}
\end{equation}

Let $\wt{\hat{\xi}}_i=\prod_j\zeta_j^{a_{ji}}$ and
${\hat{\xi}}_i=1$. Our original $QQ$-system is gauge equivalent to the
system described by the parameters $\hat{b}_{ij}, \wt{\hat{\xi}}_i,
\hat{{\xi}}_i$.  Indeed, one can set $\alpha^{(i)}=\xi_i$ and 
\begin{equation}
D^{\{i\}}=q^{b_{i_1i_2}+b_{i_2i_3}+\dots +b_{i_{k-1}i_k}-(k-1)}.
\end{equation}
\end{proof}

\section{$(G,q)$-opers and Baxter relations}    \label{Sec:DSReduction}

The space of $\g$-opers on the punctured disc may be identified with
the phase space of the Drinfeld-Sokolov reduction
\cite{Drinfeld:1985}. Likewise, the space of $(G,q)$-opers on the
punctured disc can be described in terms of the $q$-difference version
of the Drinfeld-Sokolov reduction, which was defined in
\cite{Frenkel1998,1998CMaPh.192..631S}. For example, just as
$\sl_n$-opers can be represented as $n$th order differential
operators, $(\SL(n),q)$-opers can be represented as $n$th order
$q$-difference operators. In this section, we apply results of
\cite{Frenkel1998,1998CMaPh.192..631S} to construct a system of canonical
coordinates for $q$-opers on $\P^1$ with regular singularities for all
simply connected simple complex Lie groups $G$ other than $E_6$. We
conjecture that if $G$ is simply laced, then the relations between
these canonical coordinates and the polynomials $Q^i_+(z)$ introduced
above are equivalent to the generalized Baxter $TQ$-relations
established in \cite{Frenkel:2013uda}.

\subsection{Presentations of unipotent subgroups}

Let $\wt{s}_i \in N(H)$ be a lifting of the simple reflection $w_i \in
W_G$ (not necessarily equal to the liftings $s_i$ used before).  Given
$w\in W_G$, let $\wt{w}=\wt{s}_{i_1}\dots \wt{s}_{i_k}$ of $w$ ne the
ocorresponding lifting
Suppose that $w\in W_G$ has a reduced decomposition $w=w_{i_1}\dots
w_{i_k}$ where the $w_{i_j}$'s are distinct, and consider the lifting
$\wt{w}=\wt{s}_{i_1}\dots \wt{s}_{i_k}$ of $w$.

Let $X_\a$ denote the root subgroup corresponding to $\a$.  Consider
the subgroup
$$
N_-^w=N_-\cap \wt{w}N_+ \wt{w}^{-1}.
$$
It is well-known that $N_-^{w}$ is a subgroup of $N_-$ of
dimension $\ell(w)$~\cite[Section 28]{humphreys}, which does not
depend on the choice of the lifting $\wt{w}$.  Moreover,
$$N_-^w=\prod_{\a<0,w\cdot\a>0}X_\a,$$ where the product can be taken
in any order.  In the case of a Coxeter
element, or more generally, for any Weyl group element with a reduced
decomposition consisting of distinct simple reflections, this group
can be described more explicitly.

\begin{Prop} \label{Nw} Let $w$ be a Weyl group element whose reduced
  decompositions consist of distinct simple reflections.  Then, we
  have the presentation
\begin{equation} N_-^w =\wt{s}_{i_1}X_{\a_{i_1}}\dots
    \wt{s}_{i_r}X_{\a_{i_k}}\wt{w}^{-1}.
  \end{equation}
\end{Prop}

\begin{proof} Consider the roots $\beta_j=w_{i_1}\dots
  w_{i_j}(\a_{i_j})$ for  $j=1,\ldots,k=\ell(w)$.  We will show that these
  are precisely the negative roots $\a$ such that $w^{-1}(\a)$ is
  positive.  First, observe that each $\beta_j$ is negative.  Indeed,
  $\beta_j=w_{i_1}\dots w_{i_{j-1}}(-\a_{i_j})=-\a_{i_j}+\gamma_j$,
  where $\gamma_j$ is a linear combination of the $\a_{i_s}$'s for
  $1\le s\le j-1$.  Since the $\a_{i_s}$'s are distinct, the
  coefficient of $\a_{i_j}$ is negative, so $\beta_j$ is negative.
  Next, we show that the $\beta_j$'s are distinct.  Suppose
  $\beta_j=\beta_s$ with $s>j$.  Then, we have
  $\a_{i_j}=w_{i_{j+1}}\dots w_{i_s}(\a_{i_s})$.  However, the root on
  the right is negative by the argument above, contradicting the fact
  that $\a_{i_j}$ is positive.  Finally, $w^{-1}\beta_j=w_{i_k}\dots
  w_{i_{j+1}}(\a_{i_j})$, which is positive.

    We conclude that $N_-^w=X_{\beta_1}\dots X_{\beta_k}$.  The result
    now follows from the fact that 
    \begin{align*} X_{\beta_1}\dots
      X_{\beta_k}&=\wt{s}_{i_1}X_{\a_{i_1}}\wt{s}_{i_1}^{-1}
      \wt{s}_{i_1}\wt{s}_{i_2}X_{\a_{i_1}}\wt{s}_{i_1}^{-1}
      \wt{s}_{i_2}^{-1}\dots\wt{w}X_{\a_{i_k}}
      \wt{w}^{-1}\\&=\wt{s}_{i_1}X_{\a_{i_1}}\dots
    \wt{s}_{i_r}X_{\a_{i_k}}\wt{w}^{-1}.
    \end{align*}
\end{proof}

\subsection{Canonical coordinates on $q$-opers}

It follows from Definition \ref{d:regsing} that a $(G,q)$-oper
with regular singularities determined by the nonconstant polynomials
$\{ \Lambda_i(z) \}_{i=1,\ldots,r}$ is an equivalence class of
$q$-connections of the form
$$
A(z)= n'(z)\prod_i(\Lambda_i(z)^{\check{\alpha}_i} \, s_i)n(z), \qquad
n(z), n'(z)\in N_-(z),
$$
under the action of
$q$-gauge transformations by elements $N_-(z)$. We now describe a set
of canonical representatives for these equivalence classes, using
Theorem 3.1 from
\cite{1998CMaPh.192..631S}. Since this theorem was proved in
\cite{1998CMaPh.192..631S} for all simple Lie groups except $E_6$, in
the rest of this section we will assume that $G \neq E_6$.

\begin{Thm}    \label{thm:qds}
For every $A(z) \in n'(z)\prod_i(\Lambda_i(z)^{\check{\alpha}_i} \,
  s_i)n(z)$, there exists a unique $u'(z)\in N_-(z)$ such that
$$
u'(qz) A(z){u'}^{-1}(z) \in N^s_-(z)
\prod_i(\Lambda_i(z)^{\check{\alpha}_i} \, s_i).
$$
Moreover, there exist unique $T_i(z)\in\C(z)$ such
that
\begin{equation}    \label{Ti}
u'(qz) A(z){u'}^{-1}(z)=\prod_i
  \Big[\Lambda_i(z)^{\check\alpha_i}s_i e^{-T_i(z)e_i}\Big].
\end{equation}
  As usual, the order in the product is determined by the Coxeter
  element $c$.
\end{Thm}
\begin{proof}  Note that $\wt{s}_i=\Lambda_i(z)^{\check{\alpha}_i}
  \, s_i$ is a lifting of $w_i$ and
  $\wt{s}=\prod_i(\Lambda_i(z)^{\check{\alpha}_i} \, s_i)$ is the
  corresponding lifting of $c$.
Let $\wt{N}_-^c=N_-\cap c N_- c^{-1}$.  It is a standard fact that
$N_-=N_-^c\wt{N}_-^c$~\cite[Section 28]{humphreys}.  If we
write $n'(z)=u(z)v(z)$ with $u(z)\in N_-^c$ and $v(z)\in
\wt{N}_-^c$, then
$$
A(z)=u(z) \, \wt{s}\; \wt{s}^{-1} \, v(z) \, \wt{s}
\, n(z)=u(z) \, \wt{s} \, \wt{n}(z),
$$ for some  $\wt{n}(z)\in N_-(z)$.
Theorem 3.1 of \cite{1998CMaPh.192..631S} now applies to show that
there is a unique $u'(z)$ such that
$$
u'(qz) A(z){u'}^{-1}(z)\in
N_-^c(z)\wt{s}.
$$
It then follows from Proposition \ref{Nw} that
$$
u'(qz)
A(z){u'}^{-1}(z)=\prod_i\wt{s}_{i}x_{\a_i},
$$
where $x_{\a_i}\in
X_{\a_i}(z)$.  Since $x_{\a_i}=e^{-T_i(z)e_i}$ for a unique
$T_i(z)\in\C(z)$, we obtain the desired factorization.
\end{proof}

\begin{Rem}
  Note that the same proof applies if we consider general meromorphic
  $(G,q)$-opers instead of those with regular singularities, so that
  the polynomials $\Lambda_i(z)$ are replaced by arbitrary non-zero
  rational functions $\phi_i(z)$ as in \eqref{qop1}. 
\end{Rem}

Thus, we obtain a system of canonical representatives \eqref{Ti} for
$(G,q)$-opers with regular singularities determined by $\{
\Lambda_i(z) \}_{i=1,\ldots,r}$ and hence, a system of canonical
coordinates $(T_1(z),\ldots,T_r(z))$ on the space of these
$(G,q)$-opers.

\subsection{Generalized Baxter relations}

Now, suppose that our $(G,q)$-oper has the structure of a nondegenerate
$Z$-twisted Miura-Pl\"ucker $(G,q)$-oper. According to Theorem
\ref{BAE}, it is then uniquely determined by the polynomials $\{
Q^i_+(z) \}_{i=1,\ldots,r}$ giving a nondegenerate solution of the
Bethe Ansatz equations \eqref{bethe}. Therefore, we can express the
coordinates $T_j(z)$ of this $(G,q)$-oper in terms of $\{ Q^i_+(z)
\}_{i=1,\ldots,r}$ and $\{ \Lambda_i(z) \}_{i=1,\ldots,r}$.

Consider the case $G=\SL_2$. Choose a canonical representative
$$
\wt{A}(z) = \begin{pmatrix} 0 & \Lambda(z) \\ -\Lambda(z)^{-1} & \Lambda(z)
  T(z) \end{pmatrix}
$$ for a $q$-gauge equivalence class as in \eqref{Ti}.  Suppose that
the $q$-connection $A(z)$ given by formula \eqref{acon1} is in this
equivalence class. Then, there exists $u(z) \in \C(z)$ such that
\begin{multline}    \label{minus}
\begin{pmatrix} 1 & 0 \\ u(zq) & 1 \end{pmatrix} \begin{pmatrix}
  0 & \Lambda(z) \\ -\Lambda(z)^{-1} & \Lambda(z)
  T(z) \end{pmatrix} \begin{pmatrix} 1 & 0 \\ -u(z) & 1 \end{pmatrix}
\\ = \begin{pmatrix} \zeta Q_+(qz)Q^{-1}_+(z)&\Lambda(z)\\
   0 &  \zeta^{-1} Q_+^{-1}(qz)Q_+(z)
 \end{pmatrix}.
\end{multline}
Indeed, $u(z)$ is uniquely determined by the equation $u(z) =
-\zeta\frac{Q_+(zq)}{Q_+(z)} \Lambda(z)^{-1}$. Substituting this into
equation \eqref{minus} gives
\begin{equation}    \label{TQ}
T(z) = \zeta \frac{1}{\Lambda(zq)} \frac{Q_+(zq^2)}{Q_+(zq)} + \zeta^{-1}
\frac{1}{\Lambda(z)} \frac{Q_+(z)}{Q_+(zq)}.
\end{equation}

This equation, which we have now obtained on the $q$DE side of the
$q$DE/IM correspondence also makes perfect sense on the IM side; it is
a version of the celebrated {\em Baxter $TQ$-relation} for the XXZ
model. It relates the eigenvalues of the Baxter $Q$-operator and the
eigenvalues of the transfer-matrix $T(z)$ of the two-dimensional
evaluation representation of $U_{q'} \wh\sl_2$, suitably normalized,
acting on a finite-dimensional representation $V$ of $U_{q'} \wh\sl_2$
(here $q'=q^{-2}$); see Example 5.13(i) of \cite{Frenkel:2013uda}.
Note that up to multiplicative shifts by powers of $q$, $\Lambda(z)$
is the product of Drinfeld polynomials of the irreducible factors of
$V$. Moreover, the Bethe Ansatz equations \eqref{BAE sl2} follow from
formula \eqref{TQ} if we assume that $T(z)$ does not have poles at the
points $z=w_kq$, where $\{ w_k \}$ are the zeros of $Q_+(z)$.

The fact that the relation between the canonical coordinates for
$q$-opers and $Q_+(z)$ is equivalent to the Baxter $TQ$-relation was
known before. Indeed, it was shown in \cite{Frenkel1998} that the
$q$-analogue of the Miura transformation arising from the
$q$-Drinfeld-Sokolov reduction of $\SL(2)$ coincides with the formula
for the $q$-character of the fundamental representation of $U_{q'}
\wh\sl_2$. On the other hand, Baxter's $TQ$-relation can also be
interpreted as this $q$-character, in which we substitute the ratio of
the shifts of the Baxter polynomial just as in formula \eqref{TQ}, see
\cite{Frenkel:2013uda}.

In fact, an analogue of the last statement for an arbitrary quantum
affine algebra was conjectured in \cite{Frenkel:ls} and proved in
\cite{Frenkel:2013uda}.  On the other hand,
according to the results and conjectures of \cite{FR:w,Frenkel:ls},
the link between the $q$-Miura transformation arising from
$q$-Drinfeld-Sokolov reduction and the $q$-character homomorphism is
expected to hold for all simply laced groups $G$.

We thus formulate the following conjecture. Recall (see
\cite{Frenkel:2013uda}) that the
generalized Baxter $TQ$-relation expresses the eigenvalues of the
transfer-matrix of the $j$th fundamental representation of
$U_{q'}\ghat$ (suitably normalized) on the space of states determined
by the Drinfeld polynomials $\{ \Lambda_i(z) \}_{i=1,\ldots,r}$ in
terms of the generalized Baxter polynomials $\{ Q^i_+(z)
\}_{i=1,\ldots,r}$ (up to multiplicative shifts by powers of $q$).
(See \cite{Frenkel:2013uda} for the definition of the generalized
Baxter polynomials.)

\begin{Con}
  Suppose that $G$ is simply laced. For any nondegenerate $Z$-twisted
  Miura-Pl\"ucker $(G,q)$-oper, the formula expressing its canonical
  coordinate $T_j(z)$ in terms of the corresponding polynomials $\{
  Q^i_+(z) \}_{i=1,\ldots,r}$ and $\{ \Lambda_i(z) \}_{i=1,\ldots,r}$
  coincides with the generalized Baxter $TQ$-relation.
\end{Con}

If $G$ is non-simply laced, the picture becomes more complicated. (See
Conjecture 3, Section 6.3 and Appendix B in \cite{FR:w}.) We hope to
return to it in the future.

\bibliography{cpn1}
\end{document}